\documentclass{amsart}
\usepackage[english]{babel} 
\usepackage[utf8]{inputenc} 
\usepackage{amscd,amsmath,amsthm,amssymb}
\usepackage{enumerate}
\usepackage{tikz}
\usepackage[all]{xy}
\usetikzlibrary{cd}
\usepackage[colorlinks=true, linkcolor=blue, citecolor=blue, pagebackref]{hyperref}

\newtheorem{theorem}{Theorem}[section]
\newtheorem{lemma}[theorem]{Lemma}
\newtheorem{proposition}[theorem]{Proposition}
\newtheorem{corollary}[theorem]{Corollary}

\theoremstyle{definition}
\newtheorem{definition}{Definition}[section]
\newtheorem{remark}{Remark}[section]

\newenvironment{acknowledgement}{\par\addvspace{17pt}\small\rm
\trivlist\item[\hskip\labelsep{\it Acknowledgements.}]}
{\endtrivlist\addvspace{6pt}}

\renewcommand{\bar}[1]{\overline{#1}}

\newcommand{\pallino}{{\scriptscriptstyle\bullet}}

\newcommand{\anchor}{\bigstar}

\newcommand{\car}{\operatorname{char}}
\newcommand{\Der}{\operatorname{Der}}

\newcommand{\Fun}{\operatorname{Fun}}
\newcommand{\Def}{\operatorname{Def}}
\newcommand{\MC}{\operatorname{MC}}
\newcommand{\Ob}{\operatorname{Ob}}

\newcommand{\Spec}{\operatorname{Spec}}
\newcommand{\colim}{\operatorname{colim}}
\newcommand{\cdgamz}{\mathbf{CDGA}_{\mathbb{K}}^{\leq 0}}
\newcommand{\Id}{\operatorname{Id}}
\newcommand{\id}{\operatorname{Id}}

\newcommand{\K}{\mathbb{K}}
\newcommand{\Z}{\mathbb{Z}}
\newcommand{\N}{\mathbb{N}}

\newcommand{\Oh}{\mathcal{O}}

\newcommand{\sB}{\mathcal{B}}
\newcommand{\sC}{\mathcal{C}}
\newcommand{\sD}{\mathcal{D}}
\newcommand{\sF}{\mathcal{F}}

\newcommand{\sN}{\mathcal{N}}

\newcommand{\sW}{\mathcal{W}}

\newcommand{\CDGA}{\mathbf{CDGA}}
\newcommand{\Art}{\mathbf{Art}}
\newcommand{\Alg}{\mathbf{Alg}}
\newcommand{\Set}{\mathbf{Set}}
\newcommand{\bM}{\mathbf{M}}

\begin{document}

\title{On deformations of diagrams of commutative algebras}
\author{Emma Lepri}
\address{\newline
Universit\`a degli studi di Roma La Sapienza,\hfill\newline
Dipartimento di Matematica \lq\lq Guido
Castelnuovo\rq\rq,\hfill\newline
P.le Aldo Moro 5,
I-00185 Roma, Italy.\medskip}
\email{emma.lepri3@gmail.com}

\author{Marco Manetti}
\email{manetti@mat.uniroma1.it}
\urladdr{www.mat.uniroma1.it/people/manetti/}

\date{20 february 2019}


\begin{abstract} In this paper  we study classical deformations of diagrams of commutative algebras over a field of characteristic 0. In particular we determine several homotopy classes of 
DG-Lie algebras, each one of them controlling this above deformation problem: the first homotopy type is described in terms of the projective
model structure on the category of diagrams of differential graded algebras, the others in terms of the Reedy model structure on truncated Bousfield-Kan approximations.\newline%
\phantom{aa} The first half of the paper contains  an elementary introduction to the projective model structure on the category of commutative differential graded algebras, while the second half 
is devoted to the main results.\end{abstract}

\maketitle

%\tableofcontents

\section{Introduction}

Let $\K$ be a fixed field, $S$ a Noetherian commutative $\K$-algebra and  $X=\Spec(S)$ the associated affine scheme.
It is well known that every deformation of $X$, in the category of schemes over $\K$, is affine, hence the deformation theory of $X$ is the same of the deformation theory of $S$ inside the category $\Alg_{\K}$ of unitary commutative $\K$-algebras. 
Similarly, the deformation theory of a separated Noetherian scheme $X$ over $\K$ is the same as the deformation theory of a diagram in $\Alg_{\K}$. More precisely,  if $\sN$ is the nerve of an affine open cover $\{U_i\}$ of $X$, it is not difficult to prove 
that the deformations of $X$ (up to isomorphism) are the same as the deformations (up to isomorphism) of the diagram
\[ S_\pallino\colon \sN\to \Alg_{\K},\qquad S_{i_0,\ldots,i_n}=\Gamma(U_{i_0}\cap\cdots\cap U_{i_n},\Oh_X)\,,\] 
where $\sN$ is considered as a poset and as a small category in the obvious way, inside the category of diagrams
$\Fun(\sN,\Alg_{\K})=\{\sN\to \Alg_{\K}\}$. 

It is  well known (see e.g. \cite{H04,roma,MM}) that,  if $\K$ has characteristic 0, 
then every  (commutative) deformation of an algebra $S\in \Alg_{\K}$ is isomorphic to $H^0(R')$, 
where $R'$ is obtained by perturbing the differential of a fixed Tate resolution  $R\to S$ \cite{tate}.
This easily implies that the deformations of $S$ are controlled, in the sense of \cite{ManettiDGLA}, by the differential graded Lie algebra of derivations of $R$. A short introduction to differential graded algebras is given here in 
Section~\ref{sec.dgca}.

It is possible to prove that the above strategy generalises to arbitrary Noetherian separated schemes, where Tate resolution is replaced by the algebraic analogue of Palamodov's resolvent \cite{MM2,Pal}. This is possible because the nerve $\sN$  
of a covering is a direct Reedy category, i.e., there exists a degree function $\deg\colon \sN\to \N$ such that every non identity arrow increases degree. 

The aim of this paper is to study deformations of diagrams $\sD\to \Alg_{\K}$ for a general small category $\sD$. The first result is to extend the above strategy by detecting what is the correct notion of Tate resolution of a diagram (= the correct notion of Palamodov's resolvent for  a diagram). 
In doing this it is extremely convenient  to work in the framework of model structures, briefly recalled in Section~\ref{sec.model_intro}. 
 
The category $\Alg_{\K}$ can be considered in an obvious way as a full subcategory of 
$\CDGA_{\K}$ (resp.: $\CDGA_{\K}^{\le 0}$), 
the category of commutative differential graded algebras (resp.: in non-positive degrees). 

By a classical result of Bousfield and Gugenheim \cite{Bous} the category $\CDGA_{\K}$ admits a model structure where weak equivalences are the quasi-isomorphisms and fibrations are the surjective maps, cf. \cite{GM}. 

The category $\CDGA_{\K}^{\le 0}$ carries  a similar model structure, 
where weak equivalences are the quasi-isomorphisms and fibrations are the surjective maps in negative degrees. Due to the lack of appropriate references, 
in Section~\ref{sec.model_alg} we provide an elementary proof of this fact, based on the properties of free and semifree extensions.

Section~\ref{sec.derivations} is devoted to some technical lemmas that are probably well known to experts. In Section~\ref{sec.deformations} we prove the main result of this paper (Theorem~\ref{thm.main}), namely that the deformation theory of a diagram  
$S_\pallino\colon \sD\to \Alg_{\K}$ is controlled by the differential graded Lie algebra of derivations of a cofibrant replacement of $S_\pallino$ in the model category of diagrams 
$\Fun(\sD,\CDGA_{\K}^{\le 0})$, equipped with the projective model structure.

Unfortunately, for general index categories $\sD$, cofibrant replacements in the projective model structure are difficult to describe from the constructive point of view. For this reason, in the last 
sections we propose a different approach by describing a countable family of functors between small categories (Definition~\ref{def.approximation})
\[   \epsilon_k\colon N(\sD)_{\le k}\to \sD,\qquad k=2,3,\ldots,\infty,\]
such that for every $k$ in the above range:
\begin{enumerate}

\item every diagram $S_\pallino\colon \sD\to \Alg_{\K}$ has the same isomorphism classes of deformations as  $S_\pallino\circ \epsilon_k$;\smallskip

\item $N(\sD)_{\le k}$ is a Reedy category (see Section~\ref{sec.Reedy}) and 
the projective model structure  on the category $\Fun(N(\sD)_{\le k},\CDGA_{\K}^{\le 0})$  is the same as the Reedy model structure, hence with cofibrations described constructively in terms of latching objects and cofibrations in $\CDGA_{\K}^{\le 0}$.
\end{enumerate}

In our construction the functor $\epsilon_{\infty}$ is the forgetful functor from the simplex category of $\sD$ (see Section~\ref{sec.simplex}), and $\epsilon_{k}$ is its restriction to the full subcategory of $p$-simplexes, with $p\le k$. 
The composition  map 
\[ \Fun(\sD,\CDGA_{\K}^{\le 0})\xrightarrow{\;-\circ\epsilon_{\infty}\;}\Fun(N(\sD)_{\le \infty},\CDGA_{\K}^{\le 0})\]
is called Bousfield-Kan approximation and plays an important role in the homotopy theory of diagrams \cite{CS}.

Putting together all the above facts, the main result of this paper is:

\begin{theorem}[=Theorem~\ref{thm.main}+Corollary~\ref{cor.main}] 
 Let $\sD$ be a small category,  
$S_\pallino\colon \sD\to \mathbf{Alg}_{\K}$ a diagram of unitary commutative algebras.

\begin{enumerate} 
\item Let $R_\pallino\to S_\pallino$ be a cofibrant replacement in 
$\Fun(\sD,\CDGA_{\K}^{\le 0})$ with respect to the projective model structure. Then 
the DG-Lie algebra $L=\Der_{\K}^*(R_\pallino,R_\pallino)$ controls the deformations of 
$S_\pallino$.

\item For every $k=2,\ldots,\infty$, let $\epsilon_k\colon N(\sD)_{\le k}\to \sD$ be  the functor defined in \ref{def.approximation} and let  
$R_{\pallino k}\to S_\pallino\circ \epsilon_k $ be a Reedy cofibrant replacement in 
$\Fun(N(\sD)_{\le k},\CDGA_{\K}^{\le 0})$.  
Then 
the DG-Lie algebra $L_k=\Der_{\K}^*(R_{\pallino k},R_{\pallino k})$ controls the deformations of 
$S_\pallino$.
\end{enumerate}
\end{theorem}

%It is worth to point out that the DG-Lie algebras $L,L_2,L_3,\ldots, L_{\infty}$ may have  different homotopy types. This is not completely surprising since we only deal with classical deformations, whereas a classic (non derived) deformation problem does not determine uniquely a homotopy class of DG-Lie algebras.

\subsection*{Notation and setup}

Throughout this paper we will work over a fixed field $\K$ of characteristic 0. 
Unless otherwise specified, every (graded) vector space is assumed over $\K$ and the symbol $\otimes$ denotes the tensor product over $\K$. 
If $V=\oplus_{n\in\Z}V^n$ is a graded vector space, we denote by $\bar{a}$ the degree of a non-zero homogeneous element $a$: in other words $\bar{a}=n$ whenever $a\not=0$ and $a\in V^n$. 
It is implicitly assumed that if a mathematical formula contains the degree symbols 
$\bar{a},\bar{b},\ldots$ then all the elements $a,b,\ldots$ involved are homogeneous and different from $0$.
As usual, for every complex of vector spaces $V$, we shall denote by $Z^n(V), B^n(V)$ and $H^n(V)$ the space of $n$-cocycles, the space of $n$-coboundaries and the $n$th cohomology group, respectively.
We denote by $\Set$ the category of sets, by $\mathbf{Grp}$ the category of groups, 
by $\Alg_{\K}$ the category of unitary commutative $\K$-algebras and by   
$\Art_{\K}\subset \Alg_{\K}$ the full subcategory of local Artin algebras with residue field $\K$.
Finally, in order to avoid an excessive length we assume that the reader has a basic knowledge of differential graded Lie algebras and of the associated deformation functors: for instance, the papers 
\cite{ManettiDGLA,ManettiSeattle} contain everything needed for the comprehension of this paper.

\bigskip
\section{Commutative Differential Graded  Algebras}
\label{sec.dgca}

In the first four sections of this paper we shall give a short survey, addressed to a wide mathematical audience, of some homotopical algebra that we use in the second part of the paper. 
We begin by recalling the definition and the first properties of unitary commutative differential graded algebras (DG-algebras for short) over $\K$.

\begin{definition}\label{def.dgalgebras}
A unitary commutative graded algebra is a graded vector space 
$A=\bigoplus_{n \in \mathbb{Z}}A^n$ with a product $A^i \times A^j \longrightarrow A^{i+j}$ which is $\mathbb{K}$-linear, associative and graded commutative, i.e., such that 
$ab=(-1)^{\bar{a}\,\bar{b}}ba$ for every $a,b\in A$.
Moreover there exists a unit  $1 \in A^0$ such that $1a=a1=a$ for every $a\in A$.
\end{definition}

A morphism of unitary commutative graded  algebras is a morphism of  graded vector spaces 
that commutes with products and preserves the units. We denote by $\mathbf{CGA}_{\K}$ the category of 
unitary commutative graded algebras. In the above definition it is allowed that $1=0$, and this happens if and only if $A=0$.

The usual construction of polynomials extends without difficulties to the graded case.  Given a unitary commutative graded algebra $A$ and 
a set $\{x_i\}$, $i\in I$, of indeterminates, each one equipped with a degree $\overline{x_i} \in \mathbb{Z}$,  
the polynomial algebra $A[\{x_i\}]$ is defined as the graded vector space generated by the monomials in  $x_i$ with coefficients in $A$, subject to the relations $x_ix_j= (-1)^{\bar{x_i}\,\bar{x_j}}x_jx_i$ and $ax_i=(-1)^{\bar{x_i}\,\bar{a}}x_ia$, $a\in A$.
For instance, if $\bar{x}=0$ and $\bar{y}=1$, then $xy=yx$, $y^2=0$ and therefore $\K[x,y]=\K[x]\oplus \K[x]y$.

Given $A\in  \mathbf{CGA}_{\K}$, a derivation of degree $k$ of $A$ is a linear map 
$\alpha\colon A\to A$ such that $\alpha(A^n)\subset A^{n+k}$ for every $n$, satisfying the  
\emph{(graded) Leibniz identity}: 
\[\alpha(ab)=\alpha(a)b+(-1)^{k\overline{a}}a\alpha(b)\,.\]  
The vector space of  derivations of degree $k$ is denoted $\Der_{\K}^k(A,A)$.
  
If $A=\K[\{x_i\}]$, by the Leibniz identity every derivation  $\alpha\in \Der_{\K}^k(A,A)$ is uniquely defined by the values $\alpha(x_i)\in A^{\bar{x_i}+k}$.

  \begin{definition}
  A commutative differential graded  algebra (DG-algebra for short) 
  is a graded commutative algebra $A$ equipped with a 
  derivation $d\in \Der_{\K}^1(A,A)$, called \emph{differential},  such that $d^2=0$. In other words:
  \begin{enumerate}
  \item $d(A^n) \subseteq A^{n+1}$,
  \item $d^2=0$,
  \item (Graded Leibniz identity) $d(ab)=d(a)b+(-1)^{\overline{a}}ad(b)$.
  \end{enumerate}
A morphism of commutative differential graded algebras is a morphism of commutative graded algebras that commutes with differentials. 
 \end{definition}
 
We denote by $\mathbf{CDGA}_{\K}$ the category of 
commutative differential graded algebras. Notice that $d(1)=0$, and that $\K$ and $0$ are 
respectively the initial and the final object in the category $\mathbf{CDGA}_{\K}$. It is easy to see 
that this category is complete and cocomplete.

\begin{definition}[Free extensions]\label{def.free}%\label{ex.free}
Let $A \in \mathbf{CDGA}_{\K}$ and $x_i$, $i \in I$, a set of indeterminates of degree $\overline{x_i} \in \mathbb{Z}$. Consider a parallel set of indeterminates $dx_i$, with 
$\overline{dx_i}=\overline{x_i}+1$ and the 
polynomial extension  $A\to A[\{x_i,dx_i\}]$. 
The differential $d$ on $A$  can be extended to a differential on $A[\{x_i,dx_i\}]$ by setting $d(x_i)=dx_i$ and $d(dx_i)=0$. 
\end{definition}

The name free extension is motivated by the following property: for every morphism $f\colon A\to B$ in 
$\mathbf{CDGA}_{\K}$ and every subset $\{b_i\}\subset B$ with $b_i\in B^{\bar{x_i}}$ for every $i$, there exists a unique morphism of DG-algebras $g\colon A[\{x_i,dx_i\}]\to B$ extending $f$ and such that $g(x_i)=b_i$ for every $i$. Clearly $g(dx_i)=d(b_i)$.

\begin{lemma}\label{lem.freequasiiso}
Every free extension of DG-algebras is a quasi-isomorphism, i.e., the inclusion 
$A\to  A[\{x_i,dx_i\}]$ induces an isomorphism in cohomology.
\end{lemma}

\begin{proof}
Since every element of $A[\{x_i,dx_i\}]$  is a polynomial in a finite number of indeterminates, we can assume the set of variables  finite, say $x_1,\ldots,x_n$,  and proceed by induction on $n$.
Therefore it is sufficient to show that the inclusion $A\to A[x,dx]$ is a quasi-isomorphism.

If $x$ has even degree, then $(dx)^2=0$ and every homogeneous element of the quotient $A[x,dx]/A$ is of type
\[ v=\sum_{i=0}^n x^{i+1}a_i+x^i dx\, b_i,\qquad a_i,b_i\in  A\,.\]
If $dv=0$ then $d(b_i)=(i+1)a_i$ and therefore (notice the assumption $\car(\mathbb{K})=0$)
\[ v=d\left(\sum_{i=0}^n \frac{x^{i+1}}{i+1}\,b_i\right)\,.\]

If $x$ has odd degree, then $x^2=0$ and every homogeneous element of the quotient $A[x,dx]/A$ is of type
\[ u=\sum_{i=0}^n (dx)^{i+1}\, a_i+x (dx)^i b_i,\qquad a_i,b_i\in  A\,.\]
If $du=0$ then  $da_i+b_i=0$  for every $i$, and we can write
\[ u=d\left(\sum_{i=0}^m x(dx)^i a_i\right)\,.\]
Thus we have proved that $A[x,dx]/A$ is an acyclic complex of vector spaces.
\end{proof}

\begin{definition} Let $f\colon A\to B$ be a morphism in  $\CDGA_{\K}^{\le 0}$, with $A=A^0\in \Alg_{\K}$. We shall say that $f$ is \emph{flat}, or that $B$ is a flat $A$-algebra if $B$ is a complex of 
flat $A$-modules.
\end{definition}

Clearly the above definition  extends the usual notion of flat morphism of algebras. It is worth pointing out that there also exists a good notion of flatness for every morphism in  $\CDGA_{\K}^{\le 0}$ \cite{MM}.

For every $A\in \CDGA_{\K}^{\le 0}$ we shall denote by $\CDGA_{A}^{\le 0}$ the undercategory of maps $A\to B$: the morphisms in $\CDGA_{A}^{\le 0}$ are the commutative triangles.  The following  lemma is completely standard, see e.g. \cite[Lemma A.4 and Theorem A.10]{Sernesi}.

\begin{lemma}\label{lem.nakayama} 
Let $A\in \Art_{\K}$ and let $f\colon B\to C$ be a morphism in 
$\CDGA_{A}^{\le 0}$: 
\begin{enumerate}

\item if  $C$ is flat over $A$ and  the induced map 
$B\otimes_A\K\to C\otimes_A\K$ is an isomorphism, then $f$ is also an isomorphism;

\item if  $B,C$ are flat $A$-algebras and  the induced map 
$B\otimes_A\K\to C\otimes_A\K$ is a quasi-isomorphism, then $f$ is also a quasi-isomorphism;

\item if $B$ is flat over $A$ and $H^i(B\otimes_A\K)=0$ for every $i<0$, then 
$H^0(B)$ is a flat $A$-algebra and the natural map $H^0(B)\to H^0(B\otimes_A\K)$ induces an 
isomorphism $H^0(B)\otimes_A\K=H^0(B\otimes_A\K)$. 
\end{enumerate} 
\end{lemma}

\bigskip
\section{A very short introduction to  model structures}
\label{sec.model_intro}

We briefly recall the definition of model category and some few basic results about them; the reader may consult \cite{Hov,Hir} for a deeper and more complete exposition of the subject. Throughout this section 
$\bM$ will denote a fixed category.

\begin{definition}[Lifting properties] Consider two morphisms $i\colon A\to B$, $f\colon C\to D$ in $\bM$.
If for every solid commutative diagram 
\begin{center}
  \begin{tikzcd}
   A \arrow[r] \arrow[d, "i"'] & C \arrow[d, "f"] \\
   B \arrow[r] \arrow[ru, dotted] & D
  \end{tikzcd}
 \end{center}  
 there exists the dotted arrow that makes both triangles commute, we shall say that the map $i$ has the left lifting property with respect to $f$, and the map $f$ has the right lifting property with respect to $i$.
\end{definition}

For instance, in the category of sets, every injective map $i$ has the left lifting property with respect to any 
surjective map $f$. The same holds in the category of vector spaces.

\begin{definition}[Retracts]
A morphism  $f$ in $\bM$  is called a retract of a morphism $g$ in $\bM$ if there exists a commutative diagram:
\begin{center}
  \begin{tikzcd}
  A \arrow[r] \arrow[d, "f"] \arrow[rr, bend left, "\Id_A"] & B  \arrow[r] \arrow[d, "g"] & A \arrow[d, "f"] \\
  C  \arrow[r] \arrow[rr, bend right, "\Id_C"']& D  \arrow[r] & C 
  \end{tikzcd}
\end{center}
\end{definition}

\begin{definition}\label{mod}
A model structure on $\bM$  is the data of three classes of maps: weak equivalences, 
fibrations and cofibrations, which satisfy the following axioms:
 
 \begin{enumerate}[(M1)]
  \item (2-out-of-3) If $f$ and $g$ are morphisms in $\bM$ such that the composition $gf$ is defined, and two out of the three $f$, $g$ and $gf$ are weak equivalences, so is the third.
  \item (Retracts) If $f$ and $g$ are maps in $\bM$ such that $f$ is a retract of $g$, and $g$ is a weak equivalence, a cofibration or a fibration, then so is $f$.
  \item (Lifting) A trivial fibration is map which is both a fibration and a weak equivalence; a trivial cofibration is map which is both a cofibration and a weak equivalence.
  \begin{enumerate}
	 \item Trivial fibrations have the right lifting property with respect to cofibrations.\vspace{1pt}
	 \item Trivial cofibrations have the left lifting property with respect to fibrations.
	\end{enumerate}
 \item(Factorisation) Every morphism  $g$ in $\bM$ admits  two  factorisations: 
	\begin{description}
     \item[(CW,\,F):] ~$g=qj$, where $j$ is a trivial cofibration and  $q$ is a fibration,\vspace{1pt}
	 \item[(C,\,FW):] ~$g=pi$, where $i$ is a cofibration and $p$ is a trivial fibration.
	\end{description}
\end{enumerate}
 \end{definition}
 
 \begin{definition}
 A model category is a complete and cocomplete category equipped with a model structure. 
 \end{definition}

In particular every model category has an initial object $\amalg_{\emptyset}$ and a final object ${\scriptstyle\prod}_{\emptyset}$; an object $X$ is called cofibrant in the morphism $\amalg_{\emptyset}\to X$ is a cofibration; it is called fibrant if the morphism $X\to {\scriptstyle\prod}_{\emptyset}$ is a fibration.  A \emph{cofibrant replacement} of an object $Y$ is a trivial fibration $X\to Y$ with $X$ cofibrant. The factorisation axiom guarantees that cofibrant replacements always exist.

For notational simplicity we shall denote by $\sW,\sF$ and $\sC$ the classes of weak equivalences, fibrations and cofibrations, respectively. We shall denote by $\sF\sW=\sF\cap \sW$ the class of trivial fibrations and by  $\sC\sW=\sC\cap \sW$ the class of trivial cofibrations.

\begin{lemma}\label{ret-sol}
 If $j$ has the left (right) lifting property with respect to $f$, and $i$ is a retract of $j$, then $i$ has the left (right) lifting property with respect to $f$. 
\end{lemma}

For a proof, see \cite[7.2.8]{Hir}.

\begin{proposition}[Retract Argument]\label{retractargument}
Let $g$ be a map which can be factored as $g=pi$
\begin{enumerate}
 \item If $g$ has the left lifting property with respect to $p$
then $g$ is a retract of $i$.
 \item If $g$ has the right lifting property with respect to $i$
 then $g$ is a retract of $p$.
\end{enumerate}
\end{proposition}

For a proof, see \cite[1.1.9]{Hov}.

\begin{lemma}\label{lem.liftingdefinition} 
Let $\bM$ be a model category:
\begin{enumerate}
\item A map in $\bM$ that has the left lifting property with respect to all trivial fibrations is a cofibration.
\item A map in $\bM$ that has the right lifting property with respect to all trivial cofibrations is a fibration.
\end{enumerate}
\end{lemma}

For a proof, see \cite[1.1.10]{Hov}.

\begin{remark}
It follows from the previous lemma that isomorphisms belong to all three classes of maps. Furthermore, two of the three classes $\sW, \sC, \sF$ determine the third. Pay attention to the fact that, for example, if
$\sW$ and $\sF$ are two classes satisfying (M1) and (M2), in general they do not extend to a model structure.  
\end{remark}

The following  lemma is clear.

\begin{lemma}
If $h=gf$ and both $f,g$ have the left (right) lifting property with respect to $p$, then $h$ has the left (right) lifting property with respect to $p$.
\end{lemma}

\begin{remark}
The previous lemmas show that the three classes of fibrations, cofibrations and weak equivalences are closed by composition.
\end{remark}

Model categories were introduced by Quillen \cite{QuillenHA} under the name of (complete and cocomplete) closed model categories. Nowadays many authors (e.g. \cite{Hir,Hov}) assume that the (C,FW) and (CW,F) factorisations are functorial. Since in algebraic geometry it is often important to resolve minimally algebraic structures, we prefer here to adopt the original Quillen's assumption, and require only the existence of factorisations.

\subsection{Pre-model structures} In several concrete cases, a convenient way to describe model structures is in terms of \emph{pre-model structures}.
Since two out of the three classes $\sW, \sC, \sF $ determine the third, it is typical to try to construct a model structure by establishing two out of the three classes and seeing whether they extend to a model structure. Often it happens that two classes have an easy description and the third is more complicated, but it has a nice subclass sufficiently large to ensure axiom (M4). The notion of left pre-model structure applies  when one has fixed the weak equivalences, fibrations and two more classes,  as in the next definition.

\begin{definition}\label{def.leftpremodel}
A \emph{left pre-model structure} on $\bM$ is the data of four classes of maps: $\mathcal{W},\mathcal{F}, \mathcal{C'},\mathcal{CW'}$ such that:  
\begin{enumerate}
    \item (2-out-of-3) The maps in $\mathcal{W}$ satisfy the 2-out-of-3 property;
    \item (Retracts) The classes $\mathcal{W}$ and $\mathcal{F}$ are closed under retracts;
    \item $\mathcal{CW'} \subseteq \mathcal{W}$;
    \item (Lifting) The maps in $\mathcal{C'}$ have the left lifting property with respect to the maps in $\mathcal{F}\cap \mathcal{W}$; the maps in $\mathcal{CW'}$ have the left lifting property with respect to the maps in $\mathcal{F}$;
    \item (Factorisation) Every map $g$ in $\bM$ has two  factorisations: 
    \begin{enumerate}
        \item ~$g=qj$, where $j$ is in $\mathcal{C'}$ and  $q$ is in $\mathcal{F}\cap \mathcal{W}$,
	     \item ~$g=pi$, where $i$ is in $\mathcal{CW'}$ and $p$ is in $\mathcal{F}$.
	 \end{enumerate}
\end{enumerate}
\end{definition}

\begin{theorem}\label{thm.premodel}
Given a left pre-model structure $\mathcal{W},\mathcal{F}, \mathcal{C'},\mathcal{CW'} $  there exists a unique model structure  where the weak equivalences are the maps in $\mathcal{W}$ and the fibrations are the maps in $\mathcal{F}$. Notably, the cofibrations are the retracts of $\mathcal{C'}$, and the trivial cofibrations are the retracts of $\mathcal{CW'}$.
\end{theorem}

\begin{proof}
We set $\sC$ the retracts of $\sC'$ and check that $\sC, \sF, \sW $ satisfy the model category axioms (Definition \ref{mod}). Axioms (M1) and (M2) follow immediately from the definition of pre-model structure and of $\sC$. As above, for notational simplicity we shall denote by $\sC\sW=\sC\cap \sW$ and 
$\sF\sW=\sF\cap \sW$.

We first show that $\sC\sW'\subset \sC$. Let $g\colon A\to B$ be a morphism in $\sC\sW'$ and consider a factorisation  $g\colon A\xrightarrow{\sC'} X \xrightarrow{\sF\sW} B$. Since 
$g$ has the   left lifting property with respect to maps in $\sF\sW$ it follows by the retract argument that $g$ is a retract of an element of $\sC'$. Since $\sC\sW'\subset \sW$ by assumption, we have $\sC\sW'\subset \sC\sW$.

We now show that every map in $\sC\sW$ is a retract  of a  map in $\sC\sW'$. 
Let $f\colon A \longrightarrow B$ be a map in $\sC\sW$, using the factorisation axiom $f\colon A\xrightarrow{\sC\sW'} X \xrightarrow{\sF} B$, and by the 2-out-of-3 axiom $X \longrightarrow B$ is in $\sF\sW$. Therefore by the lifting axiom and the retract argument (\ref{retractargument}), $A \longrightarrow B$ is a retract of $A \longrightarrow X$, so it is the retract of a map in $\sC\sW'$.

By definition of pre-model structure, maps in $\sC'$ have the left lifting property with respect to maps in $\sF\sW$; then by Lemma \ref{ret-sol} maps in $\sC$ have the left lifting property with respect to maps in $\sF\sW$. Similarly, maps in $\sC\sW'$ have the left lifting property with respect to maps in $\sF$, so we have that maps in $\sC\sW$ have the left lifting property with respect to maps in $\sF$.

The factorisation axiom (M4) is clear since we have already proved that $\sC\subset \sC'$ and 
$\sC\sW' \subseteq \sC\sW$.
\end{proof}

It is plain that one can also give the analogous notion of right pre-model structure, 
simply working in the opposite category and exchanging the role of $\sC$ and $\sF$. 
Finally, the reader should be aware that some authors use the  name of pre-model structure for a completely different concept.

\bigskip
\section{Model structure on DG-algebras}
\label{sec.model_alg}

It is well known that the category $\CDGA_{\K}$ admits a model structure where weak equivalences are the quasi-isomorphisms and fibrations are the surjective maps \cite{Bous,GM}. 
Consequently, by Lemma~\ref{lem.liftingdefinition} the cofibrations are the morphisms that have the left lifting property with respect to the class of surjective quasi-isomorphisms. 

Similarly, the category 
$\mathbf{CDGA}_{\K}^{\le 0}$ of DG-algebras concentrated in non-positive degree
 admits a model structure where weak equivalences are the quasi-isomorphisms and fibrations are the surjective maps in negative degree. It is worth noticing that the existence of the model structure on $\mathbf{CDGA}_{\K}^{\le 0}$ is an 
immediate consequence of \cite[Proposition 4.5.4.6]{Lur2} applied to the standard (cofibrantly generated)
model structure on the category of non-positively graded DG-vector spaces. Moreover,
by Lurie's result also follows that the model structure on $\mathbf{CDGA}_{\K}^{\le 0}$ is combinatorial and 
cofibrantly generated.

In this section, following the ideas of  \cite{Jar},  we give an elementary proof of the 
above mentioned model structure
on $\mathbf{CDGA}_{\K}^{\le 0}$, which relies on the notion of semifree extension.

\begin{definition}[Semifree extension]\label{def.semifree}
  Let $A \in \cdgamz$, $I$ be a set, and let $x_i$, $i \in I$ be indeterminates of non-positive degree  $\overline{x_i} \in \mathbb{Z}_{\le 0}$.
  Any inclusion of DG-algebras of type
  \[ A  \longrightarrow A[\{x_i\}],\]
regardless of the differential on $A[\{x_i\}]$,  is called a \emph{semifree extension}.  
\end{definition}

Recall that a differential on $A[\{x_i\}]$ is determined by the differential on $A$ and by  the values $d(x_i)$. Every free extension is also semifree. 

\begin{theorem}\label{thm.modelcdga}
There exists a model structure on the category $\cdgamz$, where weak equivalences are  
the quasi-isomorphisms and fibrations are the maps surjective in negative degree. 
Moreover: 
\begin{enumerate}
\item cofibrations are the retracts of semifree extensions, 

\item  trivial cofibrations are the retracts of free extensions,

\item trivial fibrations are the surjective quasi-isomorphisms.
\end{enumerate}
\end{theorem}

The proof that every trivial fibration is surjective is a simple argument in basic homological algebra.
In fact, if  $f\colon A \longrightarrow B$ is a quasi-isomorphism which is surjective in negative degree, 
for every $x \in B^0$, since $dx=0$  and $f\colon H^0(A)\to H^0(B)$ is bijective, 
there exist $y \in A^0$ and $z\in B^{-1}$ such that $f(y)=x+dz$. Since 
$f\colon A^{-1} \longrightarrow B^{-1}$ is surjective there exists $u \in A^{-1}$ such that $f(u)=z$ 
and therefore $x= f(y-du)$.

Now the  proof of Theorem~\ref{thm.modelcdga} follows, according to Theorem~\ref{thm.premodel}, from the fact that the four classes:
\begin{enumerate}
    \item $\mathcal{W}$  quasi-isomorphisms,
    \item $\mathcal{F}$ maps surjective in negative degree,
    \item $\mathcal{C'}$ semifree extensions,
    \item $\mathcal{CW'}$ free extensions.
\end{enumerate}
form a left pre-model structure. We have already proved in Lemma~\ref{lem.freequasiiso} that $\sC\sW'\subset \sW$.
The 2-out-of-3 axiom for $\sW$ is clear, and 
the retract axiom for $\sW,\sF$ is also obvious: the retract of a injective (surjective) map is also injective (surjective).

\begin{proposition}\label{sol1}
The maps in $\mathcal{C'}$, i.e., the semifree extensions,  have the left lifting property with respect to all trivial fibrations.
\end{proposition}

\begin{proof}
Consider the following solid commutative diagram
 \[
  \begin{tikzcd}
   A \arrow[r, "\alpha"] \arrow[d, "f"'] & C \arrow[d, "g"] \\
   A[\{x_i\}] \arrow[ru, dotted, "\gamma"] \arrow[r, "\beta"] & D
  \end{tikzcd}
 \]
where $g$ is  a trivial fibration and  $f$ a semifree extension.
For every integer $n\ge -1$ consider the DG-subalgebra of  $A[\{x_i\}]$
\[ A_n=A[\{ x_i\mid \bar{x_i}\ge -n\}]\,.\]
We have $A_{-1}=A$, $\cup_{n}A_n=A[\{x_i\}]$ and therefore, setting $\gamma_{-1}=\alpha$ it is sufficient to prove by induction that for every $n\ge 0$ we have a commutative diagram
 \[
  \begin{tikzcd}
   A_{n-1} \arrow[r, "\gamma_{n-1}"] \arrow[d] & C \arrow[d, "g"] \\
   A_n \arrow[ru,  "\gamma_n"] \arrow[r, "\beta"] & D
  \end{tikzcd}
 \]
where the left vertical arrow is the inclusion $A_{n-1}\subset A_n$.
If $n=0$, then for every $x_i$ with  $\overline{x_i}=0$ we have  $dx_i=0$. Since $g$ is surjective  there exists $c_i \in C$ such that $g(c_i)= \beta (x_i)$. We define $\gamma_0$ by setting
 $\gamma_0(x_i)= c_i$.

Assume now $n>0$ and $\gamma_{n-1}$ already defined. 
If $\overline{x_i} = -n$, then  $\overline{d(x_i)}= - n+1$ and therefore $dx_i\in A_{n-1}$. 
We have that
\[ d\gamma_{n-1}(dx_i)=\gamma_{n-1}(d^2x_i)=0,\qquad  g\gamma_{n-1}(dx_i)= \beta(dx_i)= d\beta(x_i)\]
and since $g$ is injective in cohomology we have $\gamma_{n-1}(dx_i)= dy_i$ with $y_i \in C$. 
Setting  $z_i= \beta(x_i) - g(y_i)$
we have 
\[dz_i=\beta(dx_i)-g(dy_i)=\beta(dx_i)-g\gamma_{n-1}(dx_i)=0\,.\]
Since $g$ is a surjective quasi-isomorphism there exists $c_i \in C$ such that 
$dc_i=0$ and $g(c_i)=z_i$. We can now define $\gamma_n(x_i)=y_i + c_i$.
\end{proof}

\begin{proposition}
Maps in $\mathcal{CW'}$ have the left lifting property with respect to all fibrations.
\end{proposition}

\begin{proof} Consider the following solid commutative diagram
\[
  \begin{tikzcd}
   A \arrow[r, "\alpha"] \arrow[d, "i"'] & C \arrow[d, "g"] \\
   A[\{x_i, dx_i\}] \arrow[ru, dotted, "h"] \arrow[r, "\beta"] & D
  \end{tikzcd}
 \]
with $g$ surjective in negative degrees. If $A=0$ there is nothing to prove; otherwise, since 
$A[\{x_i, dx_i\}]\in \mathbf{CDGA}^{\le0}_{\K}$, every  $x_i$ has  negative degree and there exists $c_i \in C$ such that $g(c_i)= \beta(x_i)$. We set $h(x_i)=c_i$, $h(dx_i)=dc_i$, and $h|_A=\alpha$. 
\end{proof}

\begin{proposition}\label{prop.CW-F}
Every map in $\mathbf{CDGA}^{\leq0}_{\K}$ can be factored as a free extension  followed by a fibration. 
\end{proposition}

\begin{proof}
Let $f: A \rightarrow B$ be a map in $\mathbf{CDGA}^{\leq0}_{\K}$. For every homogeneous element $b \in B$ of strictly negative degree we add two indeterminates $x_b$ and $dx_b$ to $A$, with $x_b$ of degree $\overline{b}$ and $dx_b$ of degree $\overline{b}+1$, obtaining the free extension
$$ A \overset{i}\longrightarrow A[\{x_b, dx_b\}]$$
We define $\pi: A[\{x_b, dx_b\}] \rightarrow B$ in the following way:
\begin{enumerate}
\item $\pi$ is equal to $f$ on $A$,
\item $\pi(x_b)= b$,
\item $\pi(dx_b)= db$.
\end{enumerate}
The map $\pi$ is obviously a fibration and the composition 
$$ A \overset{i}\longrightarrow A[\{x_b, dx_b\}] \overset{\pi}\longrightarrow B$$ is equal to $f$. 
\end{proof}

\begin{proposition}
Every map in $\mathbf{CDGA}^{\leq0}_{\K}$ can be factored as a semifree extension  followed by a trivial fibration.

\end{proposition}
\begin{proof}
Since  semifree extensions are closed by composition, 
according to  Proposition~\ref{prop.CW-F} it is sufficient to prove that every morphism 
$f\colon A \rightarrow B$ factors as a composition of a semifree extension and a quasi-isomorphism.

We use the differential graded analog of  the classical argument about the existence of 
Tate-Tyurina resolutions: 
we construct recursively a countable sequence of semifree extensions
$$ A=A_0 \subset A_1 \subset \cdots \subset A_n \subset \cdots $$ 
together with morphisms of DG-algebras
$f_n\colon A_n \rightarrow B$ such that $f_0=f$ and:
\begin{enumerate}
\item $A_{n+1}=A_n[\{x_i\}]$, with $\bar{x_i}=-n$;
\item $f_{n+1}$ extends $f_n$;
\item $f_n\colon Z^i(A_n)\to  Z^i(B)$ is surjective for every $i>-n$;
\item $f_n\colon H^i(A_n)\to  H^i(B)$ is  bijective for every 
$i>-n+1$.
\end{enumerate}

Let $\{b_i\}$, $i\in I$, be a set of generators of $B^0$ as a $A^0$-algebra; then we may define 
\[A_1=A[\{x_i\}],\qquad \bar{x_i}=0,\quad dx_i=0,\quad f_1(x_i)=b_i\,.\]

Assume now $n>0$ and $f_n\colon A_n\to B$ defined. 
By choosing a suitable set of generators of $Z^{-n}(B)$ as $A^0_n$-module we can first consider a factorisation  
$f_n\colon A_n\subset C\xrightarrow{g}B$ such that 
\[C=A_n[\{x_i\}],\qquad \bar{x_i}=-n,\quad dx_i=0,\quad f_n(x_i)\in Z^{-n}(B),\]
and such that $g\colon Z^{-n}(C)\to Z^{-n}(B)$   is surjective.
If $g\colon H^{-n+1}(C)\to H^{-n+1}(B)$ is bijective we can define $A_{n+1}=C$ and $f_{n+1}=g$.
Otherwise let $\{c_i\}$ be a set of elements in $Z^{-n+1}(C)$ whose cohomology classes generate
the kernel of $g\colon H^{-n+1}(C)\to H^{-n+1}(B)$,
choose elements $b_i\in B^{-n}$ such that $db_i=g(c_i)$  
 and consider the factorisation
\[A_{n+1}=C[\{x_i\}],\qquad \bar{x_i}=-n,\quad dx_i=c_i,\quad f_{n+1}(x_i)=b_i.\]
It is easy to verify that the map $f_{n+1}\colon A_{n+1}\to B$ has the required properties.
Finally, since $H^i(A_n)=H^i(A_{n+1})$ for every $i>-n+1$, the colimit of the sequence
 $f_{n+1}\colon A_{n+1}\to B$ gives the required factorisation.
\end{proof}

\begin{remark}\label{rem.samecofibrations} 
Let $f$ be a morphism in $\CDGA_{\K}^{\le 0}$. 
We have already proved that $f$ is a trivial  fibration in $\CDGA_{\K}^{\le 0}$ if and only if 
it is a trivial  fibration in $\CDGA_{\K}$. 
Since the truncation functor 
\[\tau\colon  \CDGA_{\K}\to \CDGA_{\K}^{\le 0},\qquad (\tau A)^n=\begin{cases}A^n&\quad n<0,\\
Z^0(A)&\quad n=0,\\
0&\quad n>0,\end{cases}\]
is right adjoint to the faithful natural inclusion $\CDGA_{\K}^{\le 0}\subset \CDGA_{\K}$ and 
preserves trivial fibrations, by Lemma~\ref{lem.liftingdefinition} it follows that  $f$ is cofibration in $\CDGA_{\K}^{\le 0}$ if and only if 
it is a cofibration in $\CDGA_{\K}$. 
\end{remark}

The notions of semifree extension and left pre-model structure apply to many other contexts, for instance cochain complexes over a commutative ring, DG-algebras, DG-Lie algebras etc.: full details will appear in the forthcoming thesis of the first author.

\bigskip
\section{Modules and derivations}
\label{sec.derivations}

Let $(A,d_A)$ be in $\mathbf{CDGA}_{\K}$, an $A$-module is a
differential graded vector space $(M,d_M)$ together with an associative and 
distributive  $\mathbb{K}$-linear left multiplication
map $A\times M\to M$, with the
properties:
\begin{enumerate}
\item $A^{i}M^{j}\subset M^{i+j}$,
\item $d_M(am)=d_A(a)\,m+(-1)^{\bar{a}}a\,d_M(m)$ for every $a\in A$, $m\in M$.
\end{enumerate}

A morphism of $A$-modules is a morphism of differential graded vector spaces commuting with multiplications. Since $A$ is graded commutative, we can also define an associative right multiplication map 
$M\times A\to M$ by setting $ma=(-1)^{\bar{a}\,\bar{m}}am$,  $a\in A$, $m\in
M$. Notice that $a(mb)=(am)b$ for every $a,b\in A$, $m\in M$.
 
The \emph{trivial extension} of a DG-algebra $A$ by the $A$-module $M$ is the direct sum of complexes
$A\oplus M$ equipped with the product:
\[ (a,m)(b,n)=(ab,mb+an).\]
It is immediate to see that $A\oplus M\in \mathbf{CDGA}_{\K}$, the projection $A\oplus M\to A$ is a morphism of DG-algebras and $M$ is a square-zero ideal of $A\oplus M$.

For a given graded vector space $M$ and an integer $n$ we shall denote by $M[-n]$ the same space with the degrees shifted by $-n$, namely $M[-n]^i=M^{i-n}$, and by $s^n\colon M\to M[-n]$ the tautological (bijective) map of degree $n$. In other words, $s^n$ is the essentially the identity and its only effect is changing the degree:
\[ s^n\colon M^{i}\to M[-n]^{i+n},\qquad x\mapsto s^nx\,.\]
If $M$ is an $A$-module, then $M[-n]$ is also an $A$-module, where the differential and the product are defined accordingly to the Koszul sign rule:
\[ d(s^nx)=(-1)^ns^n d(x),\qquad a(s^nx)=(-1)^{n\bar{a}}s^n(ax),\qquad (s^nx)a= s^n(xa)\,.\]

\begin{definition}
    Let $M$ be an $A$-module. A $\mathbb{K}$-linear map $\alpha\colon A \rightarrow M$ is a derivation of degree $j \in \mathbb{Z}$ if $\alpha(A^n) \subset M^{n+j}$ and it satisfies Leibniz's law: 
    $$ \alpha(ab)= \alpha(a)b + (-1)^{\overline{a}j} a \alpha(b)$$
\end{definition}

The vector space of derivations of degree $j$ from $A$ to $M$ is denoted 
$\Der^{j}_{\K}(A,M)$. The graded vector space 
$\Der^{*}_{\K}(A,M)= \bigoplus_{j \in \mathbb{Z}}\Der^{j}_{\K}(A,M)$ has a natural structure of $A$-module, 
with multiplication $(a\alpha)(x)=a(\alpha(x))$ and differential 
$(d\alpha)(x)=d(\alpha(x))-(-1)^{\bar{\alpha}}\alpha(dx)$.
Observe that for every  integer $n$ there is a natural isomorphism of $A$-modules
\[ \Der^*_{\K}(A,M[-n])\to \Der^*_{\K}(A,M)[-n]\,.\]

Every morphism of DG-algebras  
$f\colon A\to B$ induces in the natural way an $A$-module  structure  on $B$.
In this case the module of derivations will be denoted $\Der^{*}_{\K}(A,B;f)$: a $\K$-linear map 
$\alpha\colon A\to B$ is an $f$-derivation of degree $k$ if 
$\alpha(A^n)\subset B^{n+k}$ and 
$\alpha(ab)=\alpha(a)f(b)+(-1)^{k\bar{a}}f(a)\alpha(b)$.

\begin{remark} Let $f\colon A\to B$ be a morphism of DG-algebras,  $I\subset B$ a square-zero ideal
and $\pi\colon B\to B/I$ the quotient map.  
Then $I$ is  a $B/I$-module and then also an $A$-module via the morphism $\pi f$. 
It is immediate to check that if $g\colon A\to B$ is a morphism of 
graded algebras such that $\pi g=\pi f$ then $g-f\colon A\to I$ is a 
derivation of degree $0$. Conversely, if $\alpha\in \Der^0_{\K}(A,I)$, then 
$f+\alpha$ is a morphism of graded algebras, and it is a morphism of DG-algebras if and only if 
$\alpha\in Z^0(\Der^*_{\K}(A,I))$.
\end{remark}

\begin{lemma}\label{lem.derivationproperties} 
Let $A\in \CDGA_{\K}$ be a cofibrant algebra and $f\colon M\to N$ a surjective quasi-isomorphism of $A$-modules. Then the map
\[ f_*\colon \Der^*_{\K}(A,M)\to \Der^*_{\K}(A,N),\qquad \alpha\mapsto f\alpha\,,\]
is a surjective quasi-isomorphism.
\end{lemma}

\begin{proof} Since $f\colon M[-n]\to N[-n]$ is a surjective quasi-isomorphism for every integer $n$ it is sufficient to prove that: 
\begin{enumerate}
\item $f_*\colon \Der^0_{\K}(A,M)\to \Der^0_{\K}(A,N)$ is surjective;

\item $f_*\colon H^0(\Der^*_{\K}(A,M))\to H^0(\Der^*_{\K}(A,N))$ is bijective.
\end{enumerate}

Let's denote by $C(P)=P\oplus P[1]$ the mapping cone of the identity of an $A$-module $P$, with the differential defined by the formula $d(x+s^{-1}y)=dx+y-s^{-1}dy$; notice that $C(P)$ is acyclic and the natural projection
$C(P)\to P[1]$ is a morphism of $A$-modules.

For every linear map $\alpha\colon A\to P$ of degree 0 we shall denote 
\[\widetilde{\alpha}\colon A\to A\oplus C(P),\qquad 
\widetilde{\alpha}(a)=a+\alpha(a)+s^{-1}(\alpha(da)-d\alpha(a))\,.\]
It is straightforward to check that $\widetilde{\alpha}$ is a morphism of complexes and that every morphism of complexes $A\to A\oplus C(P)$ lifting the identity on $A$ is obtained this way.
Moreover, $\alpha$ is a derivation if and only if $\widetilde{\alpha}$ is a morphism in $\CDGA_{\K}$.

Since $A\oplus C(M)\to A\oplus C(N)$, $a+x+s^{-1}y\mapsto a+f(x)+s^{-1}f(y)$, is a trivial fibration, the lifting of a derivation $\alpha\in \Der^0_{\K}(A,N)$ is obtained by taking the lifting of the 
morphism of DG-algebras $\widetilde{\alpha}\colon A\to A\oplus C(N)$. This proves the first item.

If $K$ is the kernel of $f$, then we have an exact sequence of complexes 
\[ 0\to  \Der^*_{\K}(A,K)\to \Der^*_{\K}(A,M)\to \Der^*_{\K}(A,N)\to 0\]
and  in order to prove the second item it is sufficient to show that $\Der^*_{\K}(A,K)$ is acyclic.
By the shifting degree argument it is sufficient to prove that $H^{1}(\Der^*_{\K}(A,K))=0$.
Given $\beta\in Z^{1}(\Der^*_{\K}(A,K))$, the map 
\[ \widehat{\beta}\colon A\to A\oplus K[1],\qquad \widehat{\beta}(a)=a+s^{-1}\beta(a),\]
is a morphism of DG-algebras and the proof that 
$\beta\in B^{1}(\Der^*_{\K}(A,K))$ follows immediately by considering a lifting of $\widehat{\beta}$ along the 
trivial fibration $A\oplus C(K)\to A\oplus K[1]$. 

\end{proof}

\bigskip
\section{Deformations of diagrams via projective cofibrant resolutions}
\label{sec.deformations}

Throughout this section we shall denote by $\sD$ a fixed small category. For every 
category $\bM$ we shall denote by $\Fun(\sD,\bM)$ the category of diagrams 
$\sD\to \bM$. For every local Artin $\K$-algebra $A$ with residue field $\K$ we shall denote 
by $\Alg_A$ the category of unitary commutative $A$-algebras. 
For simplicity of notation, if 
\[P_\pallino\in \Fun(\sD,\Alg_A),\qquad \sD\ni a\mapsto P_a\,,\]
is a diagram of $A$-algebras and $A\to B$ is a morphism of algebras, we shall denote $P_\pallino\otimes_AB$ the diagram 
$(P_\pallino\otimes_AB)_a=P_a\otimes_AB$, $a\in \sD$.

Here we are interested in studying the deformation theory of a diagram 
$S_\pallino\colon \sD\to \mathbf{Alg}_{\K}$ of unitary commutative algebras. 

\begin{definition} A \emph{deformation} over $A\in \Art_{\K}$ of a diagram $S_\pallino\colon \sD\to \mathbf{Alg}_{\K}$ is the data of 
a diagram $S_{\pallino A}\colon \sD\to \mathbf{Alg}_{A}$ of \emph{flat} $A$-algebras and a morphism of diagrams of algebras
$\phi\colon S_{\pallino A} \to S_{\pallino}$ inducing an isomorphism $S_{\pallino A}\otimes_A\K\simeq S_{\pallino}$. 

Two deformations $\phi\colon S_{\pallino A} \to S_{\pallino}$ and $\psi\colon S_{\pallino A}' \to S_{\pallino}$ are isomorphic if there exists an isomorphism of diagrams of $A$-algebras $\eta\colon  S_{\pallino A}\to S_{\pallino A}'$ such that $\phi=\psi\eta$.
\end{definition}

It is possible to prove, see e.g. \cite[A.2]{Lur}, that for every small category $\sD$ there exist 
model structures on $\Fun(\sD,\CDGA_{\K}^{\le 0})$ and $\Fun(\sD,\CDGA_{\K})$, called 
\emph{projective model structures}, such that a morphism 
of diagrams $F\to G$ is a weak equivalence (resp.: fibration) if and only if 
$F_a\to G_a$ is a  weak equivalence (resp.: fibration) for every $a\in \sD$. 
The same argument used in Remark~\ref{rem.samecofibrations} shows that a morphism 
$f$ in $\Fun(\sD,\CDGA_{\K}^{\le 0})$ is a weak equivalence, cofibration, trivial fibration 
in $\Fun(\sD,\CDGA_{\K}^{\le 0})$ 
if and only if it is a  weak equivalence, cofibration, trivial fibration 
in $\Fun(\sD,\CDGA_{\K})$, respectively.

The notions of module and derivation extend naturally to the context of diagrams. 
For every diagram $R_{\pallino}\colon \sD\to  \CDGA_{\K}^{\le 0}$ the DG-Lie algebra of derivations is 
\[ \Der^*_{\K}(R_\pallino,R_\pallino)=\left\{\{\alpha_a\}\in \prod_{a\in \sD}\Der_{\K}^*(R_a,R_a)\mid
\alpha_bR_f=R_f\alpha_a,\; \forall\; a\xrightarrow{f}b\right\}\,.\]
It is plain that $\Der^*_{\K}(R_\pallino,R_\pallino)$ is a DG-Lie subalgebra of  
$\prod_{a\in \sD}\Der_{\K}^*(R_a,R_a)$.

An $R_\pallino$-module $M_\pallino$
is a diagram of differential graded vector spaces over $\sD$ such that 
$M_a$ is an  $R_a$-module  for every $a\in\sD$
and, for every arrow $a\xrightarrow{f}b$ in $\sD$, the map $M_f\colon M_a\to M_b$
is a morphism of $R_a$-modules, where $M_b$ is considered as a $R_a$-module via the morphism of DG-algebras 
$R_f\colon R_a\to R_b$. A morphism $g\colon M_\pallino\to N_\pallino$ of $R_\pallino$-modules is a morphism of diagrams of DG-vector spaces such that $g_a\colon M_a\to N_a$ is a morphism of $R_a$-modules for every $a\in \sD$.

%A $R_\pallino$-module $M_\pallino$
%is the data of a $R_a$-module $M_a$ for every $a\in\sD$
%and, for every arrow $a\xrightarrow{f}b$ in $\sD$, of a morphism of $R_a$-modules 
%$M_f\colon M_a\to M_b$, where $M_b$ is considered as a $R_a$-module via the morphism of DG-algebras 
%$R_f\colon R_a\to R_b$, such that $M_{gf}=M_gM_f$ for every composition  $a\xrightarrow{f}b\xrightarrow{g}c$ in $\sD$.
%
%
%A morphism $g\colon M_\pallino\to N_\pallino$ of $R_\pallino$-modules is a morphism of diagrams such that $g_a\colon M_a\to N_a$ is a morphism of $R_a$-modules for every $a\in \sD$.

The differential graded vector space of derivations is 
\[\Der_{\K}^*(R_\pallino,M_\pallino)=\left\{\{\alpha_a\}\in \prod_{a\in \sD}\Der_{\K}^*(R_a,M_a)\mid
\alpha_bR_f=M_f\alpha_a,\; \forall\; a\xrightarrow{f}b\right\}\,.\]

The same argument used in the proof of Lemma~\ref{lem.derivationproperties} works, mutatis mutandis, also for diagrams and gives the following result.

\begin{lemma}\label{lem.derivationdiagramproperties} 
Let $R_\pallino\in \Fun(\sD,\CDGA_{\K})$ be a projective cofibrant diagram and $f\colon M_\pallino\to N_\pallino$ a 
morphism of $R_\pallino$-modules such that $f_a\colon M_a\to N_a$ is a 
surjective quasi-isomorphism for every  $a\in \sD$. Then the map
\[ f_*\colon \Der^*_{\K}(R_\pallino,M_\pallino)\to \Der^*_{\K}(R_\pallino,N_\pallino),\qquad \alpha\mapsto f\alpha\,,\]
is a surjective quasi-isomorphism.
\end{lemma}

The main goal of this section is to prove the following theorem.

\begin{theorem}\label{thm.main} 
Let $\sD$ be a small category and 
$S_\pallino\colon \sD\to \mathbf{Alg}_{\K}$ a diagram of unitary commutative algebras. 
Let $R_\pallino\to S_\pallino$ be a cofibrant replacement in 
$\Fun(\sD,\CDGA_{\K}^{\le 0})$ with respect to the projective model structure. Then 
the DG-Lie algebra $\Der_{\K}^*(R_\pallino,R_\pallino)$ controls the deformations of 
$S_\pallino$.
\end{theorem} 

In other words, the functor of isomorphism classes of deformations of $S_{\pallino}$ is isomorphic to the functor of Maurer-Cartan solutions in $\Der_{\K}^*(R_\pallino,R_\pallino)$ modulus gauge equivalence. We shall prove Theorem~\ref{thm.main} after a certain number of preliminary results.
Unless otherwise specified 
we always equip the categories $\Fun(\sD,\CDGA_{\K}^{\le 0})$ and $\Fun(\sD,\CDGA_{\K})$ with the projective model structure. Therefore $R_\pallino\to S_\pallino$ is a cofibrant resolution also in the model category 
$\Fun(\sD,\CDGA_{\K})$ and we can apply Lemma~\ref{lem.derivationdiagramproperties} to the diagram $R_\pallino$.

\begin{lemma}\label{lem.lift}
Consider a commutative square of solid arrows
\[ \xymatrix{	P_\pallino\ar@{->}[r]^{g} \ar@{->}[d]_{i} & E_\pallino\ar@{->}[d]^{p} \\
C_\pallino\ar@{->}[r]_{f} \ar@{-->}[ur] & D_\pallino	} \]
in $\Fun(\sD,\CDGA_{\K}^{\le 0})$. 
If $i$ is a cofibration and $p_a\colon E_a\to D_a$ is surjective for every $a\in \sD$, 
then there exists a  lifting $\gamma\colon C_\pallino\to E_\pallino$ in the category 
of diagrams of graded algebras.

\end{lemma}

\begin{proof} Consider the contractible polynomial algebra  $\K[d^{-1}]\in\CDGA_{\K}^{\le 0}$, where 
$\overline{d^{-1}}=-1$ and $d(d^{-1})=1$, and notice that the natural inclusion 
$\alpha\colon \K\to \K[d^{-1}]$ is a morphism of DG-algebras, while the natural projection $\beta\colon \K[d^{-1}]\to \K$ is a morphism of graded algebras; moreover $\beta\alpha$ is the identity on $\K$.
Now, the morphism
\[ E_\pallino\otimes_{\K}\K[d^{-1}]\xrightarrow{p\otimes\id} D_\pallino\otimes_{\K} \K[d^{-1}] \]
is a trivial fibration, and so there exists a commutative square
\[ \xymatrix{	P_\pallino\ar@{->}[rr]^-{(\id\otimes\alpha) g} \ar@{->}[d]_{i} & & E_\pallino\otimes_{\K}\K[d^{-1}] \ar@{->}[d]^{p\otimes \id} \\
C_\pallino\ar@{->}[rr]_-{(\id\otimes\alpha) f} \ar@{->}[urr]^{\varphi} & & D_\pallino\otimes_{\K}\K[d^{-1}]	} \]
in $\Fun(\sD,\CDGA_{\K}^{\le 0})$. It is now sufficient to take $\gamma = (\id\otimes\beta)\varphi$.
\end{proof}

\begin{lemma}\label{lem.main} 
Let $A\in \Art_{\K}$ and let $N_{\pallino A}\colon \sD\to \CDGA_{A}^{\le 0}$ be a diagram of flat $A$-algebras. 
Then every cofibrant replacement $f\colon P_\pallino \to N_{\pallino A}\otimes_A\K$ 
in $\Fun(\sD,\CDGA_{\K}^{\le 0})$
lifts to an $A$-linear  differential on $P_\pallino\otimes A$ and to a trivial fibration 
$P_\pallino\otimes A\to N_{\pallino A}$ in the category 
$\Fun(\sD,\CDGA_{A}^{\le 0})$.
The above lifting is unique up to $A$-linear algebra isomorphisms of   $P_\pallino\otimes A$ lifting 
the identity on $P_\pallino$. 
\end{lemma}

\begin{proof}
{\sc Existence.}
We proceed by induction on the length of the Artin ring. Since for $A=\K$ there is nothing to prove, 
we may assume $A\in \Art_{\K}$ of length $l(A)>1$ and then there exists a non-trivial element $t\in A$ annihilated by the maximal ideal $\mathfrak{m}_A$, giving a small extension
\[ 0\to \K \xrightarrow{\;\cdot t\;} A\to B \to 0,\qquad l(B)=l(A)-1\,.\]
By induction there exist a $B$-linear differential on $P_\pallino\otimes_{\K}B$ and a commutative square in $\Fun(\sD,\CDGA_{B}^{\le 0})$:
\[ \xymatrix{	P_\pallino\otimes_{\K}B \ar@{->}[r]^{q} \ar@{->}[d] & N_{\pallino A}\otimes_AB \ar@{->}[d] 
\\
P_\pallino\ar@{->}[r]^{f} & N_{\pallino A}\otimes_A\K	} \]
In view of the embedding $\K\subset B$, the $B$-linear morphism of diagrams $q$ is uniquely determined by its restriction 
$q_{|P_\pallino}\colon P_\pallino\to N_{\pallino A}\otimes_AB$, which is a morphism in 
$\Fun(\sD,\CDGA_{\K}^{\le 0})$. By Lemma~\ref{lem.lift} we can lift $q_{|P_\pallino}$ to a morphism of diagrams of graded algebras $P_\pallino\to N_{\pallino A}$ and then we get a commutative diagram with (pointwise) exact rows
\[ \xymatrix{0\ar[r]&P_\pallino\ar[r]^t\ar[d]^f&P_\pallino\otimes_{\K}A\ar[r]^{\pi}\ar[d]^p&
P_\pallino\otimes_{\K}B\ar[r]\ar[d]^q&0\\
0\ar[r]&N_{\pallino A}\otimes_A\K\ar[r]^t&N_{\pallino A}\ar[r]&
N_{\pallino A}\otimes_{A}B\ar[r]&0}\]
Since $f$ and $q$ are trivial fibrations, to conclude the proof it is sufficient to show that there exists a lifting of the differential of $P_\pallino\otimes_{\K}B$ to an $A$-linear differential of 
$P_\pallino\otimes_{\K}A$ making $p$ a morphism of diagrams of DG-algebras.
Let $d=\{ d_a\colon P_a\to P_a\mid a\in \sD\}$ be the differential of $P_\pallino$, since the $A$-linear derivations of $P_\pallino\otimes_{\K}A$  of degree 1 lifting $d$ are of type 
$d+\eta$ with $\eta\in \Der^1_{\K}(P_\pallino,P_\pallino)\otimes\mathfrak{m}_A$, we can lift 
the differential of $P_\pallino\otimes_{\K}B$ to an $A$-linear derivation  
$\delta\colon P_\pallino\otimes_{\K}A\to P_\pallino\otimes_{\K}A$ of degree 1.

Now it is sufficient to prove that there exists a derivation 
in $\xi \in \Der^1_{\K}(P_\pallino,P_\pallino)$ such that: 
\begin{enumerate}
\item $(\delta + t\xi)^2 = 0$,
\item $p(\delta + t\xi) = d_{N_\pallino}\, p$,
\end{enumerate}
and consider $\delta + t\xi$ as the differential of $P_\pallino\otimes_{\K}A$. Since $t$ is annihilated by the maximal ideal $\mathfrak{m}_A$, the condition $(\delta + t\xi)^2 = 0$ is equivalent to 
$d\xi+\xi d=0$, i.e., the above condition (1) holds if and only if 
$\xi\in Z^1(\Der^*_{\K}(P_\pallino,P_\pallino))$.
The map  $\psi=d_{N_\pallino}\, p-p\delta$ is $A$-linear and its image is contained in 
$t(N_{\pallino A}\otimes_A\K)$, hence it factors to a derivation
$\phi\in Z^1(\Der^*_{\K}(P_\pallino,N_{\pallino A}\otimes_A\K;f))$ and the above condition (2) is equivalent to 
$\phi=f\xi$.
It is now sufficient to observe that since $f$ is a trivial fibration, by Lemma~\ref{lem.derivationdiagramproperties} the morphism 
\[ f\colon \Der^*_{\K}(P_\pallino,P_\pallino)\to \Der^*_{\K}(P_\pallino,N_{\pallino A}\otimes_A\K;f)\]
is a surjective quasi-isomorphism and therefore  
\[ f\colon Z^1(\Der^*_{\K}(P_\pallino,P_\pallino))\to Z^1(\Der^*_{\K}(P_\pallino,N_{\pallino A}\otimes_A\K;f))\]
is a surjective map.\\

{\sc Unicity.} Let $\delta,\delta'$ be two $A$ linear differentials on $P_\pallino\otimes_{\K}A$ lifting 
the differential $d$ on $P_\pallino$ and let 
\[ p\colon (P_\pallino\otimes_{\K}A,\delta)\to N,\qquad q\colon (P_\pallino\otimes_{\K}A,\delta')\to N,\]
be two morphisms in  $\Fun(\sD,\CDGA_{A}^{\le 0})$ lifting the trivial fibration 
$f\colon P_\pallino\to N_{\pallino A}\otimes_A\K$. We need to prove that there exists an isomorphism of diagrams of 
differential graded  $A$-algebras $\phi\colon (P_\pallino\otimes_{\K}A,\delta)\to (P_\pallino\otimes_{\K}A,\delta')$ such that $p=q\phi$.

By induction on the length we can assume that there exists 
an isomorphism of diagrams of differential graded 
$B$-algebras $\phi'\colon (P_\pallino\otimes_{\K}B,\delta)\to (P_\pallino\otimes_{\K}B,\delta')$ such that $p=q\phi'$. By Lemma~\ref{lem.lift} we can lift $\phi'$ to an isomorphism 
of diagrams of 
 graded  $A$-algebras 
 $\psi'\colon (P_\pallino\otimes_{\K}A,\delta)\to (P_\pallino\otimes_{\K}A,\delta')$; therefore, replacing $\delta'$ with $(\psi')^{-1}\delta'(\psi)$ and $q$ with $q\psi'$ if necessary, it is not 
 restrictive to assume $\phi'$ equal to the identity.
 The derivation $p-q\colon P_\pallino\to t(N_{\pallino A}\otimes_A\K)$ can be lifted to a derivation
 $\alpha\in \Der^0_{\K}(P_\pallino,P_\pallino)$ and then, replacing $q$ with 
 $q(\id+t\alpha)$ and $\delta'$ with $(\id-t\alpha)\delta'(\id+t\alpha)$ if necessary, it is not restrictive to assume $p=q$.
 This implies in particular that $\delta'=\delta+t\xi$, for some 
 $\xi\in Z^1(\Der^*_{\K}(P_\pallino,P_\pallino))$. Since 
 the kernel of $f\colon \Der^*_{\K}(P_\pallino,P_\pallino)\to \Der^*_{\K}(P_\pallino,N_{\pallino A}\otimes_A\K;f))$ is acyclic and 
 $p\delta=p\delta'$, we have $f\xi=0$ and therefore $\xi=[d,\alpha]$ for some 
 $\alpha\in \Der^0_{\K}(P_\pallino,P_\pallino)$ such that $f\alpha=0$.
 Now $e^{t\alpha}$ is the required isomorphism.  
\end{proof}

\begin{proof}[of Theorem~\ref{thm.main}]
We assume that the reader has a certain familiarity with the theory of deformation functors 
associated to DG-Lie algebras; the basic facts exposed in \cite{ManettiDGLA,ManettiSeattle} are sufficient for our needs.

Let $\sD$ be a small category and 
$S_\pallino\colon \sD\to \mathbf{Alg}_{\K}$ a diagram of unitary commutative algebras. 
Let $R_\pallino\to S_\pallino$ be a cofibrant replacement in 
$\Fun(\sD,\CDGA_{\K}^{\le 0})$ with respect to the projective model structure and consider the DG-Lie algebra $L=\Der_{\K}^*(R_\pallino,R_\pallino)$.

Denoting by $\Def_{S_\pallino}\colon \Art_{\K}\to \Set$ the functor of isomorphism classes of deformations we want to describe an isomorphism
\[ \phi\colon \Def_L\to \Def_{S_\pallino}\,.\] 

Denoting by $d\in \Der^1_{\K}(R_\pallino, R_\pallino)$ the differential of $R_\pallino$, 
for every $A\in \Art_{\K}$ with maximal ideal $\mathfrak{m}_A$, a Maurer-Cartan element 
\[ \xi\in \MC_L(A)=\left\{x\in L^1\otimes\mathfrak{m}_A\;\middle|\; d_Lx+\frac{1}{2}[x,x]=0
\right\}\]
is exactly a derivation $\xi\in \Der^1_{\K}(R_\pallino, R_\pallino\otimes \mathfrak{m}_A)$ such that $(R_\pallino\otimes A,d+\xi)$ is a flat diagram in $\Fun(\sD,\CDGA^{\le 0}_A)$.
Moreover $\xi,\eta\in  \MC_L(A)$ are gauge equivalent if and only if there exists an isomorphism of diagrams of DG-algebras $(R_\pallino\otimes A,d+\xi)\simeq (R_\pallino\otimes A,d+\eta)$, lifting the identity over $R_\pallino$.
According to Lemma~\ref{lem.nakayama} the map 
\[ \MC_L(A)\to \Def_{S_\pallino}(A),\qquad \xi\mapsto H^0(R_\pallino\otimes A,d+\xi),\]
is properly defined and factors to a natural transformation 
$\phi\colon \Def_L\to \Def_{S_\pallino}$. 
Finally Lemma~\ref{lem.main} implies immediately that $\phi$ is an isomorphism.
 \end{proof}

In the situation of Theorem~\ref{thm.main}, according to  Lemma~\ref{lem.derivationdiagramproperties}, the natural map 
$L=\Der^*_{\K}(R_\pallino, R_\pallino)\to \Der^*_{\K}(R_\pallino, S_\pallino)$ is a quasi-isomorphism of complexes, hence $H^i(L)=0$ for every $i<0$.
The $S_\pallino$-module $\Der^*_{\K}(R_\pallino, S_\pallino)$,  defined up to quasi-isomorphism, is called the \emph{tangent complex} of $S_\pallino$. Its cohomology groups are denoted by $T^i(S_\pallino)$. 
According to \cite{ManettiDGLA,ManettiSeattle}, an immediate consequence of Theorem~\ref{thm.main} is that the space of first order deformations
of the diagram $S_\pallino$ is 
$H^1(\Der^*_{\K}(R_\pallino, R_\pallino))=T^1(S_\pallino)$, and obstructions to deformations are contained in the space $H^2(\Der^*_{\K}(R_\pallino, R_\pallino))=T^2(S_\pallino)$.

\medskip

Although in principle Theorem~\ref{thm.main} gives a complete answer to our initial problem, 
for diagrams over a general small category $\sD$  it may be very difficult to concretely describe a cofibrant replacement, since projective cofibrations  are  described either as maps satisfying the left lifting property with respect to trivial fibrations, or 
as transfinite compositions of 
certain elementary cofibrations.

A possible strategy to overcome this difficulty is to give an explicit functor
of small categories $\epsilon\colon \sN\to \sD$ such that:

\begin{enumerate}

\item for every diagram $S_\pallino\in \Fun(\sD,\Alg_{\K})$, the deformation theory of $S_\pallino$ is the same as the deformation theory of $S_\pallino \circ \epsilon\in \Fun(\sN,\Alg_{\K})$;

\item cofibrations in $\Fun(\sN,\CDGA_{\K}^{\le 0})$ admit a  constructive description.

\end{enumerate}

In the next sections we follow this strategy by setting as $\epsilon$ a simplified version of the 
Bousfield-Kan approximation \cite{CS}. In our construction the category $\sN$ will be in particular a Reedy category (see Section~\ref{sec.Reedy}) and the  projective model structure in $\Fun(\sN,\CDGA_{\K}^{\le 0})$ 
will be the same as the Reedy model structure, hence with a simpler description of cofibrations.

\bigskip
\section{Simplex categories}
\label{sec.simplex}

 Let $\mathbf{\Delta}$ be the category with objects the finite ordinals $[n]= \{0, 1, \cdots, n\}$ and morphisms non-decreasing maps,  also known as the \emph{simplex category}. 
We denote by $\delta_k\colon [n-1]\to [n]$, and by 
$\sigma_k\colon [n+1]\to [n]$, $k=0,\ldots, n$, the usual face and degeneracy maps:
\[
\delta_{k}\colon [n-1]\to [n],
\qquad \delta_{k}(p)=\begin{cases}p&\text{ if }p<k\\
p+1&\text{ if }p\ge k\end{cases},\qquad k=0,\dots,n,
\]
\[
\sigma_{k}\colon [n+1]\to [n],
\qquad \sigma_{k}(p)=\begin{cases}p&\text{ if }p\le k\\
p-1&\text{ if }p>k\end{cases},\qquad k=0,\dots,n,
\]
They satisfy the cosimplicial identities:
\begin{align*}
 \sigma_i\sigma_j&=\sigma_j\sigma_{i+1} \quad\text{ for } i \geq j \\
\delta_i\delta_j&= \delta_{j+1}\delta_i \quad\text{ for } i \leq j \\
\sigma_i \delta_j &=
    \begin{cases}
      \delta_{j-1}\sigma_i, & \text{if}\ j > i+1 \\
      \Id, & \text{if}\ j =i, i+1 \\
      \delta_j \sigma_{i-1}, &\text{if}\ j < i.
    \end{cases}
    \end{align*}
    
We recall that a cosimplicial group is a functor $G\colon \mathbf{\Delta} \to \mathbf{Grp}$, $[n]\mapsto G_n$; in the sequel we shall need the 
following proposition, which is an easy generalisation of a well known result about cosimplicial 
groups, cf. \cite[Prop. X.4.9]{BK72}.

\begin{proposition}\label{prop.cosimpl}
Let $G$ be a cosimplicial group, let $n \geq  1$, and  $I\subseteq [n]$. 
Assume there are given elements 
$x_i\in G_n$, $i\in I$, such that  $\sigma_{i-1}x_j= \sigma_j x_i$ for all $i > j$ and $i,j \in I$. 
Then there exists $x \in G_{n+1} $ such that $\sigma_i x= x_i$ for all $i\in I$.
\end{proposition}

\begin{proof}
Writing  $I=\{i_0 < i_1 < \cdots < i_k \} \subseteq [n]$, consider 
the sequence $z_{i_0},\ldots,z_{i_k}$ defined recursively by the formula:
\[ z_{i_k}= \delta_{i_k} x_{i_k},\qquad z_{i_p}= z_{i_{p+1}} \cdot (\delta_{i_p} \sigma_{i_p} z_{i_{p+1}})^{-1} \cdot (\delta_{i_p} x_{i_p}),\quad p<k\,.\]
For later use we point out that in the construction of this sequence we have only used the group homomorphisms
\[ \sigma_i\colon G_{n+1}\to G_n,\qquad \delta_i\colon G_n\to G_{n+1},\qquad i\in I\,.\]
We claim that $x=z_{i_0}$ is the required element: we show by induction on $k-p$ that 
that $\sigma_{i_m}z_{i_p}= x_{i_m}$ for all $m \geq p$.
For $m > p$,
\begin{align*}
    \sigma_{i_m}z_{i_p} &=( \sigma_{i_m}z_{i_{p+1}}) \cdot (\sigma_{i_m}\delta_{i_p} \sigma_{i_p} z_{i_{p+1}})^{-1} \cdot (\sigma_{i_m}\delta_{i_p} x_{i_p}) \\
    &=x_{i_m} \cdot (\delta_{i_p} \sigma_{i_{m}-1} \sigma_{i_p} z_{i_{p+1}} )^{-1} \cdot (\delta_{i_p} \sigma_{i_{m}-1} x_{i_p}) \\
    &= x_{i_m} \cdot (\delta_{i_p} \sigma_{i_{p}} \sigma_{i_m} z_{i_{p+1}} )^{-1} \cdot (\delta_{i_p} \sigma_{i_{p}} x_{i_m}) \\
    &= x_{i_m} \cdot (\delta_{i_p} \sigma_{i_{p}} x_{i_m} )^{-1} \cdot (\delta_{i_p} \sigma_{i_{p}} x_{i_m}) = x_{i_m}\,,
\end{align*}
and for $m=p$
\[\begin{split}
     \sigma_{i_p}z_{i_p} &= (\sigma_{i_p}z_{i_{p+1}} )\cdot (\sigma_{i_p}\delta_{i_p} \sigma_{i_p} z_{i_{p+1}})^{-1} \cdot (\sigma_{i_p}\delta_{i_p} x_{i_p}) \\
     &= (\sigma_{i_p}z_{i_{p+1}} )\cdot (\sigma_{i_p} z_{i_{p+1}})^{-1} \cdot ( x_{i_p}) = x_{i_p}\,.
\end{split}\]
\end{proof}

The simplex category $\mathbf{\Delta}$ admits the following useful generalisation. 
Let $\sB$ be a small category and consider,
 for every $n\ge 0$,  the set  $N(\sB)_n$  of $n$-simplexes of the nerve of $\sB$: 
 every element of $N(\sB)_n$ is a string 
\[ x=[x_0\xrightarrow{\alpha_1}x_1\cdots x_{n-1}\xrightarrow{\alpha_n}x_n]\]
of $n$ morphisms of $\sB$.
The \emph{simplex category} $N(\sB)$ of $\sB$ is defined in the following way: the set of objects is the disjoint union of $N(\sB)_n$, $n\ge 0$. Given two objects
\[  x=[x_0\xrightarrow{\alpha_1}x_1\cdots x_{n-1}\xrightarrow{\alpha_n}x_n],\qquad 
y=[y_0\xrightarrow{\beta_1}y_1\cdots y_{m-1}\xrightarrow{\beta_m}y_m],\]
a morphism $f\colon x\to y$ is a monotone map $f\colon [n]\to [m]$ such that
$y_{f(i)}=x_i$ for every $i\in [n]$, and for every 
$0\le i\le n$ the morphism $\alpha_i$ is the composition of $\beta_j$, for $f(i-1)<j\le f(i)$
\[ \alpha_i\colon\, x_{i-1}=y_{f(i-1)}\xrightarrow{\beta_{f(i-1)+1}}
\cdots \xrightarrow{\qquad} y_{f(i)-1}\xrightarrow{\beta_{f(i)}}y_{f(i)}=x_i\,.\]
Notice that the equality  $f(i-1)=f(i)$ implies $x_i=x_{i-1}$ and 
$\alpha_i=\Id$. For example, 
if $[x\xrightarrow{\alpha}y\xrightarrow{\beta}z]\in N(\sB)_2$, then 
we have in the category $N(\sB)$ the following morphisms:
\[ \xymatrix{[z]\ar[d]^{\delta_0}\ar[r]^{\delta_0}&[x\xrightarrow{\beta\alpha}z]\ar[d]^{\delta_1}&\\
[y\xrightarrow{\beta}z]\ar[r]^-{\delta_0} &[x\xrightarrow{\alpha}y\xrightarrow{\beta}z]&
[x\xrightarrow{\alpha}y]\ar[l]_-{\delta_2}\\
[y\xrightarrow{\Id}y\xrightarrow{\beta}z]\ar[u]^{\sigma_0}&&\,\;[x\xrightarrow{\alpha}y\xrightarrow{\Id}y]\ar[u]^{\sigma_1}\;.}\]

Notice that the simplex category of the singleton $\sB=\{*\}$ is exactly  $\mathbf{\Delta}$.

\begin{definition} A morphism $f\colon x\to y$ in $N(\sB)$:
\[  x\in N(\sB)_n,\quad  y\in N(\sB)_m,\quad f\colon [n]\to [m],\]
is called an \emph{anchor} if $f(n)=m$.
\end{definition}

\begin{definition} 
For every $k\in \N\cup\{+\infty\}$ we shall denote by $N(\sB)_{\le k}$ the full subcategory of 
$N(\sB)$ with objects the (disjoint) union of $N(\sB)_{i}$ for $i\le k$, and by 
\[\Fun^{\anchor}(N(\sB)_{\le k},\mathbf{M})\subseteq \Fun(N(\sB)_{\le k},\mathbf{M})\] 
the full subcategory of diagrams 
$F\colon N(\sB)_{\le k}\to \mathbf{M}$ such that $F(f)$ is an isomorphism for every anchor map $f$.
\end{definition}

\begin{definition}\label{def.approximation} 
The forgetful functor $\epsilon\colon {N(\sB)} \to \sB$ is defined by setting
\[\epsilon([x_0\xrightarrow{\alpha_1}x_1\;\cdots\; x_{n-1}\xrightarrow{\alpha_n}x_n])=x_n\] 
on the objects. For any morphism 

\[ f\colon [x_0\xrightarrow{\alpha_1}x_1\,\cdots\, x_{n-1}\xrightarrow{\alpha_n}x_n]\to
[y_0\xrightarrow{\beta_1}y_1\,\cdots\, y_{m-1}\xrightarrow{\beta_m}y_m],\]
we have
\[ \epsilon(f)=\beta_m\circ\cdots\circ \beta_{f(n)+1}\colon\;  x_n=y_{f(n)}\to y_m\,.\]
\end{definition}
 
In particular, $\epsilon(f)=\Id$ for any anchor $f$.
It is clear that the composition with the functor $\epsilon$ gives, for every $k$, a natural transformation:
\begin{equation}\label{equ.epsilon}
 \epsilon^*\colon \Fun(\sB,\mathbf{M})\to \Fun^{\anchor}(N(\sB)_{\le k},\mathbf{M}),\qquad \epsilon^*(F)=F\circ \epsilon_{|N(\sB)_{\le k}}\,.
\end{equation}

If $k\ge 2$ we also have a natural transformation 
\begin{equation}\label{equ.tau}
\tau\colon \Fun^{\anchor}(N(\sB)_{\le k},\mathbf{M})\to 
\Fun(\sB,\mathbf{M})
\end{equation} 
defined in the following way: given 
$G\in \Fun^{\anchor}(N(\sB)_{\le k},\mathbf{M})$ and an object $x\in \sB$ we set
\[ \tau(G)(x)=G([x])\,.\]

Given a morphism $x\xrightarrow{\alpha}y$ in $\sB$ we have 
\[ G([x])\xrightarrow{G(\delta_1)}G([x\xrightarrow{\alpha}y])\xleftarrow{G(\delta_0)}G([y])\]
and, since $G(\delta_0)$ is an isomorphism we can define
\[ \tau(G)(\alpha)=G(\delta_0)^{-1}G(\delta_1)\colon G([x])\to G([y])\,.\]

We need to prove that $\tau(G)$ is a functor: applying $G$ to the commutative diagram   
\[ \xymatrix{[x]\ar[r]^-{\delta_1}\ar[dr]_{\Id}&[x\xrightarrow{\Id}x]\ar[d]^{\sigma_0}&[x]\ar[l]_-{\delta_0}\ar[dl]^{\Id}\\
&[x]&}\]
we prove that $\tau(G)$ preserves the identities. Given $[x\xrightarrow{\alpha}y\xrightarrow{\beta}z]\in N(\sB)_2$, applying $G$ to the commutative diagram 
\[ \xymatrix{[z]\ar[d]^{\delta_0}\ar[r]^{\delta_0}&[x\xrightarrow{\beta\alpha}z]\ar[d]^{\delta_1}&
[x]\ar[l]_-{\delta_1}\ar[d]^{\delta_1}\\
[y\xrightarrow{\beta}z]\ar[r]^-{\delta_0} &[x\xrightarrow{\alpha}y\xrightarrow{\beta}z]&
[x\xrightarrow{\alpha}y]\ar[l]_-{\delta_2}\\
&[y]\ar[ur]_{\delta_0}\ar[ul]^-{\delta_1}&}\]
we obtain $\tau(G)(\beta\alpha)=\tau(G)(\beta)\circ \tau(G)(\alpha)$.
Therefore $\tau$ is properly defined and its functoriality is clear.

\begin{proposition} 
The above functors $\epsilon^*$ and $\tau$ are equivalences of categories.
\end{proposition}

\begin{proof} It is immediate from the definition that $\tau\circ \epsilon^*$ is the identity.
On the other hand every anchor map of type 
\[ [x_n]\to [x_0\to\cdots\to x_n]\]
induces, for every $G\in \Fun^{\anchor}(N(\sB)_{\le k},\mathbf{M})$, a canonical isomorphism
\[ \epsilon^*\tau(G)([x_0\to\cdots\to x_n])=\tau(G)(x_n)=G([x_n])\xrightarrow{\;\simeq\;}
G([x_0\to\cdots\to x_n])\,.\]
\end{proof}

\begin{proposition}\label{prop.equivalenzadef} 
Let $S_\pallino\in \Fun(\sB,\Alg_{\K})$ be a diagram of commutative algebras and $k\ge 2$.
Then the isomorphism classes of deformations of $S_\pallino$ are the same as the  
isomorphism classes of deformations of $\epsilon^*S_\pallino\in \Fun(N(\sB)_{\le k},\Alg_{\K})$.
\end{proposition}

\begin{proof}
It is immediate from the definition that for every $A \in \Art_{\K}$ the equivalences of categories
\begin{center}
    \begin{tikzcd}
{\Fun(\sB, \Alg_{A}}) \arrow[r, "\epsilon^*", shift left=1] &  {\Fun^{\anchor}(N(\sB)_{\le k}, \Alg_{A})} \arrow[l, "\tau", shift left=1] 
\end{tikzcd}
\end{center}
preserve flatness.
Therefore, for every deformation $\phi\colon S_{A,\pallino} \to S_{\pallino}$ of $S_\pallino$ the map $\epsilon^*\phi\colon \epsilon^*S_{A,\pallino} \to \epsilon^*S_{\pallino}$ is a deformation of the diagram $\epsilon^*S_{\pallino}$.

%\begin{center}
%    \begin{tikzcd}
%\sB \arrow[r, "S_\pallino"] & \Alg_{\K} \\
%N(\sB) \arrow[u, "\epsilon"] \arrow[ru, "\epsilon^*S_\pallino"'] & 
%\end{tikzcd}
%\end{center}

In order to conclude the proof we only need to show that, 
if 
$R_{\pallino }\colon N(\sB)_{\le k} \to \Alg_A$ is a deformation of $\epsilon^* S_\pallino$, and  $f\colon \alpha \to \beta$ is an anchor in $N(\sB)_{\le k}$, then  $R_{f }$ is an  isomorphism. This implies that 
$R_{\pallino }\in  \Fun^{\anchor}(N(\sB)_{\le k}, \Alg_{A})$.
We have by definition that $\phi\colon R_{\pallino} \to \epsilon^* S_{\pallino}$ induces an isomorphism 
$ R_{\pallino}\otimes_A \K  \to \epsilon^* S_{\pallino}$ and, since $\epsilon^* S_{\alpha} \to \epsilon^* S_{\beta}$ is the identity, by the commutativity of the diagram
\begin{center}
    \begin{tikzcd}
    R_{\alpha}\otimes_A \K \arrow[d, "\phi", "\cong"'] \arrow[r] & R_{\beta}\otimes_A \K \arrow[d, "\phi", "\cong"']\\
    \epsilon^* S_{\alpha} \arrow[r, "\Id"'] & \epsilon^* S_{\beta}
    \end{tikzcd}
\end{center}
we obtain that $R_{\alpha}\otimes_A \K \to R_{\beta}\otimes_A \K$ is also an isomorphism. By Lemma \ref{lem.nakayama}, $R_{\alpha} \to R_{\beta}$ is an isomorphism too.

% and the same reasoning applies to the functor $\tau\colon \Fun^{\anchor}(N(\sB)_{\le k},\Alg_{\K})\to 
%\Fun(\sB,\Alg_{\K})$, so the deformation theory of $S_\pallino$ is the same as the deformation theory of $\epsilon^* S_\pallino$, provided that we show that the category $\Fun^{\anchor}(N(\sB)_{\le k},\Alg_{\K})$ is closed by deformations.
\end{proof}

\bigskip
\section{Reedy model structures}
\label{sec.Reedy}

We briefly recall the notion of Reedy category, for more details see \cite{Hir}.
We do so in view of Theorem \ref{thm.mod-ree}, which yields a model structure on the category $\Fun (\sD, \bM)$ of diagrams on a model category $\bM$ indexed by a Reedy category $\sD$.

\begin{definition}
A Reedy category is a small category $\sD$ together with two subcategories $\overrightarrow{\sD}, \overleftarrow{\sD}$, such that:
\begin{enumerate}
    \item $\Ob(\sD)=\Ob(\overrightarrow{\sD})=\Ob(\overleftarrow{\sD})$
    \item Every morphism $f $ in $\sD$ has a unique factorisation $f=gh$, where $g$ is in $\overrightarrow{\sD}$ and $h$ is in $\overleftarrow{\sD}$.
    \item There exists a function $\deg\colon \Ob(\sD) \to \mathbb{N}$ such that every non-identity morphism in $\overrightarrow{\sD}$ raises degree and every non-identity morphism in $\overleftarrow{\sD}$ lowers degree.

\end{enumerate}
\end{definition}

It is easy to see that in a Reedy category every isomorphism is an identity: if $f$ is an isomorphism 
and $f=g_1h_1$, $f^{-1}=g_2h_2$, $h_1g_2=g_3h_3$ are factorisations as in (2), then 
$g_1g_3$ must be the identity. A Reedy category is called direct if 
$\sD=\overrightarrow{\sD}$, or equivalently if $\overleftarrow{\sD}$ contains only the identities.

For instance, we have the following examples of Reedy categories:
\begin{enumerate}
    \item A category whose only morphisms are the identities is called discrete. Every discrete category is trivially a Reedy category. 
    
     \item Let $I$ be a finite poset such that there exists a function 
     $\deg\colon I \rightarrow \mathbb{N}$ such that $\deg(x) > \deg(y)$ for every $x > y$. Then  $I$ is a direct Reedy category.

    \item   The simplex category $\mathbf{\Delta}$ is a Reedy category, with $\deg([n])=n$, $\overrightarrow{\mathbf{\Delta}}$ the injective maps and $\overleftarrow{\mathbf{\Delta}}$ the surjective maps.

    \item If $\sC$ and $\sD$ are Reedy categories then so is the product $\sC \times \sD$, with $\overrightarrow{\sC \times \sD}= \overrightarrow{\sC} \times \overrightarrow{\sD}$, $\overleftarrow{\sC \times \sD}= \overleftarrow{\sC} \times \overleftarrow{\sD}$ and $\deg(c, d) = \deg(c) + \deg(d)$.
     
\end{enumerate}

If $a$ is an object of a category  $\sD$ we denote by $a\downarrow\sD$ the undercategory of maps $a \rightarrow b$ in $\sD$ and by $ {\sD} \downarrow a$ the overcategory of maps $b \rightarrow a$ in ${\sD}$.

\begin{definition} Let $\sD$ be a Reedy category and $a$ an object in $\sD$.
    \begin{enumerate}
        \item The matching category $M_a \sD$ of $\sD$ at $a$ is the full subcategory of 
        $a\downarrow\overleftarrow{\sD}$ containing all objects except the identity map of $a$.
        \item The latching category $L_a \sD$ of $\sD$ at $a$ is the full subcategory of 
        $\overrightarrow{\sD}\downarrow a$ containing all objects except the identity map of $a$.
    \end{enumerate}
\end{definition}

\begin{definition}\label{def.matchlatch}
    Let $\sD$ be a Reedy category, let $\bM$ be a complete and cocomplete category, let $X$ be a $\sD$-diagram in $\bM$, and $a$ be an object in $\sD$. For notational simplicity 
    $X$ also denotes the induced $M_a \sD$-diagram, with $X_{a \rightarrow b}= X_b$, and the induced  $L_a \sD$-diagram, with $X_{b \rightarrow a}= X_b$.
    \begin{enumerate}
        \item The matching object of $X$ at $a$ is $M_aX= \lim_{M_a\sD}X$.
        \item The latching object of $X$ at $a$ is $L_aX= \colim_{L_a\sD}X$.
    \end{enumerate}
    There are natural morphisms $L_aX \longrightarrow X$ and $X \longrightarrow M_aX$.
\end{definition}

The main use of Reedy categories originates from the following theorem: the category of diagrams in a model category indexed by a Reedy category has a model category structure.

\begin{theorem}[Reedy-Kan]\label{thm.mod-ree}
Let $\sD= (\overrightarrow{\sD}, \overleftarrow {\sD})$ be a Reedy category, and $\bM$ a model category. There is a model structure on $\Fun (\sD, \bM)$ where a map $f\colon X \longrightarrow Y$ is:
\begin{enumerate}
    \item a weak equivalence iff $X_i \longrightarrow Y_i$ is a weak equivalence for all $i \in \sD$;
    \item a fibration iff $X_i \longrightarrow M_iX \times_{M_iY} Y_i$ is a fibration for all $i \in \sD$;
    \item a cofibration iff $X_i \coprod_{L_iX} L_iY \longrightarrow Y_i$ is a cofibration for all $i \in \sD$.
\end{enumerate}
\end{theorem}

For a proof, see \cite[15.3]{Hir}.

We call Reedy weak equivalences, Reedy fibrations and Reedy cofibrations the weak equivalences, fibrations and cofibrations of this model structure, to avoid confusion with other model structures on the same category. For example, the commutative square 
\[ \xymatrix{A\ar[r]^f\ar[d]_g&B\ar[d]\\
C\ar[r]&D}\]
may be considered as a diagram over the Reedy poset of subsets of $\{0,1\}$. Then it 
is a Reedy fibrant diagram if and only if $A,B,C,D$ are fibrant objects; it is 
a Reedy cofibrant diagram if and only if $A$ is a cofibrant object and the three maps $f,g$ and 
$B\amalg_AC\to D$ are cofibrations.

We say that a map $f\colon X \rightarrow Y$ in $\Fun (\sD, \bM)$ is a pointwise weak equivalence (cofibration, fibration) if $f_i: X_i \rightarrow Y_i$ is a weak equivalence (cofibration, fibration) for all $i \in \sD$.

%If $\sD$ is a any small category (not necessarily Reedy) and $\bM$ is a cofibrantly generated model category, as is $\cdgamz$, it is well known that there exists the projective model structure on $\Fun(\sD, \bM)$, where the weak equivalences and the fibrations are defined as the pointwise weak equivalences and fibrations respectively. 

\begin{lemma}\label{lem.reedyproj} Let $\sD$ be a Reedy category. Then the Reedy model structure on
$\Fun(\sD,\CDGA_{\K}^{\le 0})$ coincides with the projective model structure if and only if every object is Reedy fibrant.
\end{lemma}

\begin{proof} By definition the two model structures have the same weak equivalences.
Let $\varphi\colon X \to Y$ be a morphism in $\Fun(\sD, \CDGA_{\K}^{\le 0})$. 
 If $\varphi$ is a Reedy fibration, then it is not difficult to prove it is a pointwise fibration (\cite{Hir}, 15.3.11); therefore  if  every pointwise fibration is a Reedy fibration then the two model structures coincide.

Assume that every object is Reedy fibrant and that
 $\varphi_i\colon X_i \longrightarrow Y_i$  is a fibration in $\CDGA_{\K}^{\le 0}$ for all $i \in \sD$; we want to prove that $\varphi$ is a Reedy fibration. 
 We denote by $c(\K)\colon \sD\to  \CDGA_{\K}^{\le 0}$ the constant diagram $i\mapsto \K$, and by 
 $K=c(\K)\times_{Y}X$ the fibre product of $\varphi$ and the initial morphism $c(\K)\to Y$. Note that $M_i c(\K)$ is concentrated in degree $0$. Since the fibre product and the matching objects are both limits, they commute by Fubini's theorem \cite[Prop. 6.2.8]{leinster}, and we have $M_i(K)= M_i(c(\K)\times_{Y}X) \cong M_i c(\K) \times_{M_i Y}M_iX$. Thus we have a morphism of cartesian squares:  
 \begin{center}
     \begin{tikzcd}[column sep=small]
K_i \arrow[ddd] \arrow[rrr] \arrow[rd] &  &  & X_i \arrow[ddd, "\varphi_i"] \arrow[ld] \\
 & M_iK \arrow[d] \arrow[r] & M_iX \arrow[d, "\widehat{\varphi_i}"] &  \\
 & M_ic(\K) \arrow[r] & M_iY &  \\
\K \arrow[rrr] \arrow[ru] &  &  & Y_i \arrow[lu]
\end{tikzcd}
 \end{center}
The map $\varphi_i\colon X_i \to Y_i $ induces $\widehat{\varphi_i}\colon M_iX \to M_iY$; 
 let $Y_i \times_{M_iY} M_iX$ be the fibre product of $\widehat{\varphi_i}$ and the natural map $Y_i\to M_iY$.
We have to show that the map $X_i \to Y_i \times_{M_iY} M_iX$ is surjective in strictly negative degree.
 Let $(\alpha, \beta) \in Y_i \times_{M_iY} M_iX$, with $\overline{\alpha}=\overline{\beta} < 0$. Without loss of generality, because of the surjectivity of $\varphi_i$, we can assume $\alpha=0$, so $\widehat{\varphi_i}(\beta)=0$, which means $\beta$ lifts to $(0, \beta) \in M_iK$, and then to $K_i$, since the map $K_i \to M_iK$ is by hypothesis a fibration. By the commutativity of the above diagram, we have the thesis.
 
Conversely, if the Reedy and projective model structures coincide, an object is Reedy fibrant if and only if it is point-wisely fibrant, and that is clearly true in $\cdgamz$. 
\end{proof}

The following result is clear.

\begin{lemma} The simplex category $N(\sB)$ is a Reedy category, where the direct subcategory 
$\overrightarrow{N(\sB)}$ is defined by injective maps and the inverse subcategory 
$\overleftarrow{N(\sB)}$ by surjective maps. The same applies to its subcategories $N(\sB)_{\le k}$.
\end{lemma}
  
In particular, when $\sB=\{*\}$ we recover the usual Reedy structure on $\mathbf{\Delta}$.
By Theorem \ref{thm.mod-ree} we have the Reedy model structure on the category $\Fun(N(\sB), {\cdgamz})$; we show that the Reedy and projective model structures on this category coincide, using Lemma $\ref{lem.reedyproj}$.

\begin{theorem} Every object in $\Fun(N(\sB)_{\le k}, {\cdgamz})$ is Reedy fibrant, and so
the Reedy model structure coincides with the projective model structure.
\end{theorem}

\begin{proof} In view of the definition of fibrations in $\cdgamz$, it is 
sufficient to show that for every $X \in \Fun(N(\sB)_{\le k}, \mathbf{Grp})$ the map $X_\alpha \to M_\alpha X$ is surjective for all $\alpha \in N(\sB)_{\le k}$.
Fix $n\le k$ and let $\alpha^{(m)} \in  N(\sB)_n$,
$$ \alpha^{(m)}= [x_0 \xrightarrow{h_1} x_1 \xrightarrow{h_2} \cdots \xrightarrow{h_n} x_n],$$
where $m$ is the number of morphisms $h_i$ equal to the identity, $0\leq m \leq n$. If $m=0$ there is nothing to prove, because the matching category of $\alpha^{(0)}$ is empty, so every $X_{\alpha^{(0)}}$ is automatically fibrant.

In case  $n=m=1$, we have 
\begin{center}
    \begin{tikzcd}
{[a\rightarrow a]} \arrow[d, "\sigma_0"] \\
{[a]} \arrow[u, "\delta_0", bend left=60] \arrow[u, "\delta_1"', bend right=60]
\end{tikzcd}
\end{center}
with $\sigma_0\delta_0= \sigma_0 \delta_1 = \Id$ by the cosimplicial identities, so $X(\sigma_0)\colon X_{\alpha^{(1)}} \to M_{\alpha^{(1)}}X$ has a section and hence is surjective. 

In general, 
assume $n\ge 2$ and $0 < m \leq n$; 
let $I=\{i \in \{0, \cdots n-1 \}\mid h_{i+1}= \Id \}$, $|I|=m$. For every $i \in I$ we have a degeneracy map $\sigma_{i}\colon \alpha^{(m)} \rightarrow \alpha_{i}^{(m-1)}$, where the $\alpha_{i}^{(m-1)}$ are suitable objects in $N(\sB)_{n-1}$. From each $\alpha_{i}^{(m-1)}$ there are $m-1$ degeneracy maps to other objects $\alpha_{i,j}^{(m-2)} \in N(\sB)_{n-2}$, and so on. For example, for $n=3,\ m=2,\ I=\{0,2\}$:
\begin{center}
    \begin{tikzcd}[column sep=small]
 & {[a \rightarrow a \rightarrow b \rightarrow b]} \arrow[ld, "\sigma_0"'] \arrow[rd, "\sigma_2"] &  \\
{[a \rightarrow b \rightarrow b]} \arrow[rd, "\sigma_1"'] &  & {[a\rightarrow a \rightarrow b]} \arrow[ld, "\sigma_0"] \\
 & {[a \rightarrow b]} & 
\end{tikzcd}
\end{center}
For every map $\sigma_i\colon \alpha^{(m)} \rightarrow \alpha_{i}^{(m-1)}$ we also have two sections $\delta_i, \delta_{i+1}\colon  \alpha_i^{(m-1)} \rightarrow \alpha^{(m)}$, so
using an identical computation to Lemma \ref{prop.cosimpl}, we have that $X_{\alpha^{(m)}}$ maps surjectively into 
$$ V:= \{ x_l \in \alpha_l^{(m-1)},\ l \in I\ |\ \sigma_{i-1} x_j = \sigma_{j} x_i \ \text{ for every } i,j \in I, \ i> j \}.$$
It is clear that the map $X_{\alpha^{(m)}}\to V$ factors through the inclusion $M_{\alpha^{(m)}} X \to V$, so $X_{\alpha^{(m)}}$ also maps surjectively into $M_{\alpha^{(m)}} X$, and we have the thesis.
\end{proof}

Finally, the following corollary is an immediate consequence of the above results.

\begin{corollary}\label{cor.main}
Let $\sB$ be a small category and 
$S_\pallino\colon \sB\to \mathbf{Alg}_{\K}$ a diagram of unitary commutative algebras. 
Let $\epsilon\colon N(\sB)_{\le k}\to \sB$ be  the functor defined in \ref{def.approximation} for some $k\ge 2$ and 
let $R_\pallino\to S_\pallino\circ \epsilon $ be a Reedy cofibrant replacement in 
$\Fun(N(\sB)_{\le k},\CDGA_{\K}^{\le 0})$. Then 
the DG-Lie algebra $\Der_{\K}^*(R_\pallino,R_\pallino)$ controls the deformations of 
$S_\pallino$.
\end{corollary}

An example of deformation problem which is naturally encoded by a diagram over a non-Reedy category is the case of deformations of pairs (algebra, idempotent), cf. \cite[Example 5.1]{MM2}.

Let $e\colon R\to R$ be an idempotent morphism of an algebra $R\in \Alg_{\K}$, then the deformations of the pair 
$(R,e)$ can be interpreted as the deformations of the diagram 
\[ R_{\pallino}\colon\quad\xymatrix{R\ar@(dr,ur)_e}\]
over the (non-Reedy) category $\sB$ that has one object $\bullet$ and two morphisms 
$\{\id,\alpha\}$, with 
$\alpha^2=\alpha$.
By Proposition~\ref{prop.equivalenzadef}, the diagrams $R_\pallino$ 
and $\epsilon^*R_\pallino\in \Fun(N(\sB)_{\le 2},\Alg_{\K})$ have the same deformation theory.
Moreover, since $\epsilon^*R_\pallino\in \Fun^{\anchor}(N(\sB)_{\le 2},\Alg_{\K})$, it is easy to see that 
the diagram  $\epsilon^*R_\pallino$ has the same deformation theory of the diagram 
\begin{equation*}\label{equ.diagramamidempotent} 
\xymatrix{R\ar@(ur,ul)[rr]^{e}\ar@(dr,dl)[rr]_{\id}&&R
\ar@(ur,ul)[rr]^{e}\ar@(dr,dl)[rr]_{\id}\ar[rr]^{\id}&&R},
\qquad R\in \bM,\quad e^2=e\,.\end{equation*}
of algebras over the following (direct Reedy) subcategory of $N(\sB)_{\le 2}$:
\begin{equation*} \xymatrix{\bullet\ar@(ur,ul)[rr]^{\delta_1}\ar@(dr,dl)[rr]_{\delta_0}&&{[\bullet\xrightarrow[\,]{\alpha}\bullet]}
\ar@(ur,ul)[rr]^{\delta_2}\ar@(dr,dl)[rr]_{\delta_0}\ar[rr]^{\delta_1}&&
{[\bullet\xrightarrow{\alpha}\bullet\xrightarrow[\,]{\alpha}\bullet]}},
\qquad \begin{matrix}\;\;\delta_0^2=\delta_1\delta_0\,,\\[3pt]
\delta_0\delta_1=\delta_2\delta_0\,,\\[3pt]
\;\;\delta_1^2=\delta_2\delta_1\,.\end{matrix}
\end{equation*}

\begin{acknowledgement} 
Both  authors thank Francesco Meazzini for useful discussions about the topics of this paper, and the anonymous referee for useful comments and remarks. M.M. is partially supported by Italian MIUR under PRIN project 2017YRA3LK \lq\lq Moduli  and Lie theory\rq\rq.
\end{acknowledgement}

\end{document}